\documentclass[10pt]{article}
\usepackage{amsmath,amsthm,amssymb,calc,graphicx}

\newtheorem{theorem}{Theorem}[section]

\newtheorem{lemma}[theorem]{Lemma}
\newtheorem{corollary}[theorem]{Corollary}
\newtheorem{definition}[theorem]{Definition}
\newtheorem{example}[theorem]{Example}
\newtheorem{remark}[theorem]{Remark}

\makeatletter   % numbering of equations
\@addtoreset{equation}{section}

\makeatother

\newcommand{\refpart}[1]{{\it (#1)}}

\newcommand{\hpgde}[1]{E{\textstyle\left(#1\right)}}

\newcommand{\hpgo}[2]{{}_{#1}\mbox{\rm F}_{\!#2}}
\newcommand{\hpg}[5]{{}_{#1}\mbox{\rm F}_{\!#2}\! \left(\left.{#3 \atop #4}\right|\, #5 \right) }

\newcommand{\AJ}[1]{\mbox{\rm #1}}

\newcommand{\ZZ}{{\Bbb Z}}

\newcommand{\QQ}{{\Bbb Q}}
\newcommand{\CC}{{\Bbb C}}
\newcommand{\PP}{{\Bbb P}}
\newcommand{\RR}{{\Bbb R}}

\newcommand{\OO}{{\cal O}}

\begin{document}

\title{Algorithms and differential relations for Belyi functions}
%{Transformations between the Heun and \\
%Gauss hypergeometric functions\\
%of the hyperbolic type}

\author{Mark van Hoeij\footnote{Department of Mathematics, 
        Florida State University, Tallahassee, Florida 32306, USA.
        E-mail: {\sf hoeij@math.fsu.edu}. Supported by NSF grant 1017880},
        Raimundas Vidunas\footnote{Lab of Geometric \& Algebraic Algorithms,
        Department of Informatics \& Telecommunications,
        National Kapodistrian University of Athens,
        Panepistimiopolis 15784 Greece . E-mail: {\sf rvidunas@gmail.com}.
        Supported by JSPS grant 20740075.}}
\date{}
\maketitle

\begin{abstract}
A tool package for computing genus 0 Belyi functions is presented,
including simplification routines, computation of moduli fields, decompositions, dessins d'enfant.
The main algorithm for computing the Belyi functions
themselves is based on implied transformations of the hypergeometric
differential equation to Fuchsian equations, preferably with few singular points.
This gives interesting differential relations between polynomial components
of a Belyi function.
\end{abstract}

\section{Introduction}
\label{sec:intro}

Although Belyi functions % \cite{Wikipedia} 
is a captivating field of research
in algebraic geometry, Galois theory and related fields,
their computation of degree over 20 is still considered hard 
even for genus 0 Belyi functions % coverings $\PP^1\to\PP^1$
\cite[Example 2.4.10]{LandoZvonkin}. % \cite[pg. 18]{zvonkin08}, 
Grothendieck doubted that  \cite[pg.~248]{GrothESQ}
{\em ``there is a uniform method for solving the problem by computer''}.
The main algorithm of this paper computes genus~0 Belyi functions 
with a given branching pattern using implied pull-back transformations 
of the hypergeometric differential equation. The algorithm is efficient when 
the transformed Fuchsian equations \cite{Wikipedia} have just a few singular points.
In the case of a possible pull-back transformation to Heun's equation 
(with 4 singularities), Belyi functions of degree 60 can be computed
within minutes using a modern computer algebra package.

We analyze computation of Belyi functions by three methods:
the most  straightforward one, comparison of expressions for their
logarithmic derivatives, and use of pull-back transformations
between differential equations. Computational complexity of each method
is well reflected by the number of {\em parasitic solutions} \cite{kreines}.
The degree 54 example in \S \ref{sec:parasite} estimates about 350 parasitic
Galois orbits of the logarithmic derivative method.  That is very far beyond
the reach of Gr\"obner basis implementations.  And yet, the algorithm in
this paper (download {\sf ComputeBelyi.mpl} from \cite{HeunURL}) computes
this example in 4 seconds.

Symbolic identification of pulled-back differential equations gives 
interesting differential relations between polynomial components 
of a Belyi functions. In particular, this comprehensively clarifies 
appearance of Chebyshev and Jacobi polynomials  
in Belyi functions; see \S \ref{ex:chebyshev}, \ref{sec:dihedral}.

Just computed Belyi functions usually have long expressions, 
especially when the definition field has a high degree. 
Then a large minimal field polynomial is typically utilized
by a computer algebra package. The problems of simplifying
the definition field and optimizing the output functions by M\"obius
transformations are briefly considered in \S \ref{sec:nfields} and
\S \ref{sec:simplify}. Additionally, \S \ref{OtherAlgorithms} discusses
a number of other computational issues: finding the composition lattice
of a given Belyi function, and getting its {\em dessins d'enfant}.
Computation of moduli fields is discussed in \S \ref{sec:conic}.

Generally, a {\em Belyi function} is a map $\varphi:S\to\PP_z^1$
from a Riemann surface $S$ to $\PP^1(\CC)$ that only branches in the fibers
$\varphi=0$, $\varphi=1$, $\varphi=\infty$. 
In this paper we consider only rational (genus $0$) Belyi functions. 
That is, we assume $S\cong\PP^1$.

The following definitions are from \cite{HyperbHeun}.
They help us to characterize the Belyi functions 
which are computed most efficiently with the use of pull-back transformations
between Fuchsian equations.
\begin{definition} \label{df:klmregular} 
Given positive integers $k,\ell,m$, a Belyi function $\varphi:\PP^1\to\PP^1$ 
is called {\em $(k,\ell,m)$-regular} if all points above $\varphi=1$ have the branching order $k$,
all points above $\varphi=0$ have the branching order $\ell$, and all points above $\varphi=\infty$
have the branching order $m$. 
\end{definition}

Examples of $(2,3,m)$-regular Belyi functions with $m\in\{3,4,5\}$
are the well-known Galois coverings $\PP^1\to\PP^1$ of degree 12, 24, 60
with the tetrahedral $A_4$, octahedral $S_4$ or icosahedral $A_5$ monodromy groups,
respectively. The Platonic solids give their dessins d'enfant \cite{platonic}.

\begin{definition} \label{df:klmnregular}
Given yet another positive integer $n$, the
 Belyi function \mbox{$\varphi:\PP^1\to\PP^1$}
is called {\em $(k,\ell,m)$-minus-$n$-regular} % (or just {\em $(k,\ell,m)$-minus-$n$})
if, with exactly $n$ exceptions, all points above $\varphi=1$ have the branching order $k$,
all points above $\varphi=0$ have the branching order $\ell$, and all points above $\varphi=\infty$
have the branching order $m$. We will also use the shorter term  {\em $(k,\ell,m)$-minus-$n$}.
% Belyi functions as well.
\end{definition}

\begin{definition}
% With integers $(k,\ell,m)$ fixed,
Let $\varphi$ be a $(k,\ell,m)$-minus-$n$-regular Belyi function. % for some $n$.
The {\em regular branchings} of $\varphi$
are the points above $z=1$ of order $k$, the points above $z=0$ of order $\ell$,
and the points above $z=\infty$ of order $m$. The other $n$ points in the three fibers 
are called {\em exceptional points} of $\varphi$. 
\end{definition}

As utilized in \cite{HyperbHeun}, the $(k,\ell,m)$-minus-$n$ Belyi functions pull-back 
hypergeometric equations with the local exponent differences $1/k,1/\ell,1/m$ to Fuchsian equations 
with $n$ singularities. When $n=3$, the pulled-back equation is normalizable to a hypergeometric
equation as well. When $n=4$, the pulled-back equation normalizable to Heun's equation;
see \S \ref{sec:pullback} for more details. In \cite{HyperbHeun},
all $(k,\ell,m)$-minus-4 Belyi functions with $1/k+1/\ell+1/m<1$ are classified.
% If we additionally assume that there must be ,
There are in total 366 Galois orbits of these Belyi functions with regular branchings 
in each of the 3 fibers. The maximal degree is 60, and the largest Galois orbit has 15 Belyi functions.
The Belyi functions without a regular branching in some fiber appear in the list \cite{HeunClass}
of {\em parametric} hypergeometric-to-Heun transformations.
Two algorithms were independently used in \cite{HyperbHeun} to compute the whole list of Belyi functions;
one of them uses Hensel modular lifting and is non-deterministic.

This article basically supplements \cite{HyperbHeun} by explaining the deterministic algorithm
used there, and a few mentioned auxiliary algorithms. 
The deterministic algorithm uses pull-back transformations to get extra algebraic equations
between undetermined coefficients of a target Belyi map. It is described in \S \ref{sec:pullback},
and a {\sf Maple} implementation is available at \cite[\sf ComputeBelyi.mpl]{HeunURL}.
The algorithm is effective to compute $(k,\ell,m)$-minus-$n$ Belyi functions with small $n$, 
particularly for $n\le 5$.
Auxiliary algorithms are described in \S \ref{sec:conic} and \S \ref{OtherAlgorithms}.

The following definitions will be convenient in presenting our algorithms, examples
and their analysis.
\begin{definition}
Let $\varphi(x)$ be a Belyi function. A {\em bachelor point} of $\varphi$ is
% is a point in one of the three exceptional fibers $\varphi=0$, $\varphi=1$, $\varphi=\infty$
such that there are no other points in the same fiber with the same branching order.
A {\em point-couple} of $\varphi$ consists of two points in the same fiber, having the
same branching order, such that there no other points in that fiber with the same branching order.
\end{definition}

\begin{definition}
A Belyi function $\varphi(x)$ is called {\em pure} if there is a fiber
where all branching orders are equal to $2$. We will call $\varphi(x)$ {\em almost pure}
if there is a fiber where all branching orders except at one point are equal to $2$. 
\end{definition}

Bachelor points must be in the three exceptional fibers $\varphi=0$, $\varphi=1$, $\varphi=\infty$,
unless the degree $d=1$. Almost pure Belyi functions have a bachelor point, clearly.
The exceptional points of $(k,\ell,m)$-minus-$n$ functions are frequently bachelor or form point-couples.

\section{Computing Belyi functions}
\label{ProgramV}

Most straightforwardly, a rational Belyi function with a given branching pattern 
is found by computing a polynomial identity  $A=B+C$ such that factorizations of $A,B,C$
reflect the given branching pattern above $\varphi=0$, $\varphi=\infty$, $\varphi=1$.
The Belyi function $\varphi(x)$ is recovered as
\begin{equation} \label{eq:belyiacb}
\varphi(x)=\frac{A}{B},\qquad \varphi(x)-1=\frac{C}{B}.
\end{equation}
% If $e$ points of branching order $f$ are prescribed above $\varphi=0$,
% then the monic polynomial $F$
% the square-free factorizations\footnote{The square-free factorization of a polynomial $F$
% is the factorization $H_1^{k_1}H_2^{k_2}$}
In particular, the polynomial identity for a $(k,\ell,m)$-minus-$n$ Belyi functions is 
\begin{equation} \label{eq:belyiabc}
P^\ell\,U=Q^m\,V+R^k\,W,
\end{equation}
where $P,Q,R$ are monic polynomials in $\CC[x]$ whose roots 
are the regular branchings, and $U,V,W$ are polynomials whose (possibly multiple) roots 
are the exceptional points. 
The polynomials $P,Q,R$ should not have multiple or common roots.
One of the polynomials $U,V,W$ may be assumed to be monic.
If there is a bachelor point, it can be assumed to be $x=\infty$ 
without extension\footnote{There is then a canonical form for a computed Belyi function.
With $x=\infty$ fixed, an affine translation $x\mapsto x+\beta$
is used to make the Galois orbit sum of some roots equal 0, and an affine scaling 
$x\mapsto\alpha x$ is used to make an appropriate quotient (with the minimal homogeneous weight)
of two non-zero coefficients equal to 1. 

While computing with unassigned $x=1$, 
it is convenient to let the affine scalings $x\mapsto\alpha x$ act on the coefficients and solutions.
The equations are then weighted-homogeneous; the weight of a polynomial coefficient 
equals its degree as a symmetric function in the polynomials roots. 
This weighting applies to the other two described methods as well, 
and is demonstrated in the example of \S \ref{ex:deg54}.}
of the moduli field. If there are two more bachelor points, they can be similarly
assumed to be $x=0$ and $x=1$ due to affine transformations. 
It is usually convenient to assign $x=\infty$ even if there are no bachelor points.

The degrees of the polynomials in (\ref{eq:belyiabc}) are set by the branching pattern
and the assignment of $x=\infty$. Their coefficients are to be determined.
The straightforward method just expands (\ref{eq:belyiabc})
and compares the $d+1$ coefficients w.r.t. $x$. 
The number of undetermined coefficients is $d+1$ as well:
this is equal to $d+2$ (the number of distinct points in the 3 fibers by the Hurwitz formula)
plus 2 (scalar factors for non-monic polynomials, say $U,W$) 
minus 3 (the degrees of freedom of M\"obius transformations).
Solving this system of equations is not practical for Belyi functions of degree $\ge12$.
One reason is numerous {\em parasitic} \cite{kreines} solutions where 
the three components in $A=B+C$ %some polynomials  in (\ref{eq:belyiabc}) 
have common roots.  % Parasitic solutions give identities $A=B+C$ with common roots. 
Undesired coalescence  of roots of $A,B$ or $C$ is frequent as well. 
Parasitic solutions may even arise in families of positive dimension.

Generally, let us refer to a system of polynomial equations for undetermined coefficients
of a Belyi function (with a given branching pattern)
as the {\em algebraic equations} or an {\em algebraic system}.

\subsection{The logarithmic derivative ansatz}
\label{sec:logdiff}

Differentiation helps to compute Belyi functions more efficiently,
as is occasionally demonstrated \cite[\S 10]{Couveignes94},
\cite{Hempel}, \cite[\S 2.2]{LandoZvonkin}, \cite[Prop.~2]{MovReiter}.
% According to Couveignes, this goes back to Friecke or Hecke?
A systematic logarithmic derivative ansatz was formulated and
extensively used by the second author \cite[\S 3]{thyperbolic}, \cite{VidunasKitaev}, 
\cite[\S 4.1]{HeunClass}. Here we recall the ansatz and prove that
it indeed leads to an algebraic system with (generally) fewer parasitic solutions.
A computational peculiarity is described by Lemma \ref{th:depeq}.
% and differential relations for some Belyi functions are analyzed in \cite{}.

The key observation is the following.
If $\varphi(x)$ is a Belyi function, then the roots of $\varphi'(x)$ are the branching points 
above $\varphi=0$ and $\varphi=1$ with the multiplicities reduced by 1.
In the setting (\ref{eq:belyiabc}) of $(k,\ell,m)$-minus-$n$ functions,
the factorized shape of logarithmic derivatives of $\varphi(x)$ and $\varphi(x)-1$
must be the following: 
\begin{align} \label{logd1}
\frac{\varphi'(x)}{\varphi(x)}=h_1 \frac{R^{k-1}\,W}{P\,Q\,F}, \qquad 
%=\ell\,\frac{Q'}{Q}+\frac{V'}{V}-k\,\frac{P'}{P}-\frac{U'}{U},\\ \label{logd2}
\frac{\varphi'(x)}{\varphi(x)-1}=h_2 \frac{P^{\ell-1}\,U}{Q\,R\,F}.
%=m\,\frac{R'}{R}+\frac{W'}{W}-k\,\frac{P'}{P}-\frac{U'}{U}.
\end{align}
Here $h_1,h_2$ are constants, and  $F$ is the product of irreducible factors 
of $U\,V\,W$, each to the power 1.
If $x=\infty$ lies above $\varphi=\infty$ then
\begin{equation} \label{eq:infless}
h_1=h_2=\mbox{[ the branching order at $x=\infty$ ]},
\end{equation}
as this is the residue of both logarithmic derivatives at $x=\infty$.
%and has the branching order $e$, then $h_1=h_2=e$.
% This is also noticed in \cite{MovasatiReiter} as well.
For $j\in\{0,1,\infty\}$, let $n_j$ denote the number of distinct points above
$\varphi=j$.

On the other hand,
\begin{align} \label{logd2}
\frac{\varphi'(x)}{\varphi(x)} & = \ell\,\frac{P'}{P}+\frac{U'}{U}-m\,\frac{Q'}{Q}-\frac{V'}{V},
\\ \label{logd2a} %\quad \mbox{etc.}
\frac{\varphi'(x)}{\varphi(x)-1} & = k\,\frac{R'}{R}+\frac{W'}{W}-m\,\frac{Q'}{Q}-\frac{V'}{V}.
\end{align}
% and similarly for the logarithmic derivative of $\varphi(x)-1$. 
We have two expressions for both logarithmic derivatives. 
A strong algebraic system is obtained by 
%comparing the first power series terms at $x=\infty$, or by 
subtracting the two expressions for $\varphi'(x)/\varphi(x)$ and similarly 
two expressions for $\varphi'(x)/(\varphi(x)-1)$, and considering 
the coefficients to $x$  in the numerators. Let us refer to the two numerators 
(after rational simplification, and disregarding multiplication by non-zero constants)
as {\em derived numerators} of $\varphi(x)$ and $\varphi(x)-1$.
They give algebraic equations of degree at most $d+1-n_j$ with $j\in\{0,1\}$, respectively.
For comparison, the straightforward method gives algebraic equations of degree up to $d$.
If $x=\infty$ lies above $\varphi=\infty$, we immediately use (\ref{eq:infless}) and 
the degree bound is $d-n_j$. The number of algebraic equations 
is then $2d-n_0-n_1=d+n_\infty-2$. Most importantly, the new algebraic system
has fewer parasitic solutions in general, as characterized by Lemma \ref{th:parasitic} 
and Corollary \ref{th:klmparasitic} below.

This logarithmic derivative ansatz does not use the location $\varphi=1$ of the third fiber.
Therefore all polynomials, including $U,V,W$ in (\ref{eq:belyiabc}), can be assumed to be monic.
The number of undetermined variables is then $(d+2)-3$. The new algebraic system is
over-determined when $n_\infty>1$. 
The polynomial identity $A=B+C$ has to be adjusted by constant multiples 
(say, to $A$ and $C$) at the latest stage. The constant multiple to $A$ can be determined 
by evaluating  $A$, $B$ at a root of $C$.

If $k=2$ in (\ref{eq:belyiabc}), the polynomial $R$ can be eliminated symbolically.
A similar symbolic elimination is possible when computing pure or almost pure Belyi functions.
This symbolic elimination can lead to useful differential expressions for polynomial
components of a Belyi map, as demonstrated in \S \ref{ex:chebyshev}, \S \ref{sec:davenport} here.

To make use of (\ref{eq:infless}), it is convenient to assign $x=\infty$ to a point above $\varphi=\infty$ 
(perhaps after permuting the three fibers) even if there are no bachelor points. 
The easiest algebraic equations % (coming from the leading terms of the numerators)
%for undetermined coefficients coming from the identifications of two logarithmic derivatives 
are typically independent, but there is always  the following dependency.
\begin{lemma} \label{th:depeq}
Let $\varphi(x)$ denote a Belyi map such that $x=\infty$ is a pole of order $p$.
Let $E'_1,E'_2,\ldots$ denote the sequence of leading coefficients (to $x$) in the derived numerator
of $\varphi(x)$, and let $E''_1,E''_2,\ldots$ denote the similar sequence in the derived numerator
of $\varphi(x)-1$. There is a dependency between the algebraic equations 
$E'_1,\ldots,E'_p$ and $E''_1,\ldots,E''_p$. 
\end{lemma}
\begin{proof}
First note the following. The algebraic equations obtained by comparing the power series
at $x=\infty$ of expressions (\ref{logd1}), (\ref{logd2}) for $\varphi'(x)/\varphi(x)$
are derived from $E'_1,E'_2,\ldots$ {\em orderly} by combining previous (so far) members
of $E'_1,E'_2,\ldots$. That is the effect of dividing the derived numerator of $\varphi(x)$
by a denominator. The considered below rational manipulation of 
% expressions for the logarithmic derivatives of $\varphi(x)$ and $\varphi(x)-1$ 
power series at $x=\infty$ leads to a sequence of equations that is a similar {\em orderly} 
consequence of $E'_1,E'_2,\ldots$ and $E''_1,E''_2,\ldots$.

Let $d$ denote the degree of $\varphi(x)$. In the setting (\ref{eq:belyiacb}),
we have \mbox{$\deg B=d-p$}, $\deg A=\deg C=d$. The logarithmic derivative ansatz gives
\begin{equation} \label{eq:ldiffabc}
\frac{A'}{A}-\frac{B'}{B}=\frac{pC}{F}, \qquad
\frac{C'}{C}-\frac{B'}{B}=\frac{pA}{F},
\end{equation}
where $F$ is the product of irreducible factors of $A\,B\,C$, each to the power 1.
We eliminate the right-hand sides after multiplying the first equation by $A$ 
and the second equation by $C$. This leads to
\begin{equation} \label{eq:acb}
\frac{(A-C)'}{A-C}=\frac{B'}{B}.
\end{equation}
%The residues of both sides are equal, hence $\deg (A-C)=\deg B$. 
Therefore $A-C$ is equal to $B$ up to a constant 
multiple\footnote{This tells us that the algebraic equations of the straightforward method 
of expanding $A=B+C$ are implied by the logarithmic derivative ansatz,
except for ignoring a leading coefficient. In particular, the coincidence of the
leading $p$ terms of $A,C$ follows from (\ref{eq:belyiacb}). It also follows that
considering a third logarithmic derivative does not add independent equations
between undetermined coefficients.
For example, the logarithmic derivative of $\varphi(x)/(\varphi(x)-1)$ is the difference of 
left-hand sides in (\ref{eq:ldiffabc}), and it is the difference of the right-hand sides 
due to the established $A=B+C$. }.
The equations for undetermined coefficients from the power series in (\ref{eq:acb})
are {\em orderly} implications of $E'_1,E'_2,\ldots$ and $E''_1,E''_2,\ldots$. 
In particular, the equations $E'_1,\ldots,E'_{p-1}$ and $E''_1,\ldots,E''_{p-1}$
imply that the first $p$ terms of $A$ and $C$ (starting from $x^d$) coincide.
But the equations $E'_p$, $E''_p$ add the tautology $d-p=d-p$ in the next power series 
term of (\ref{eq:acb}).
\end{proof}

Now we characterize parasitic solutions of the logarithmic derivative ansatz.
\begin{lemma} \label{th:parasitic}
For a parasitic solution $A=B+C$ of the straightforward computation from $(\ref{eq:belyiacb})$,
let $F$ denote the product of irreducible factors of $A\,B\,C$ each to the power $1$.
Let $G=\mbox{\rm gcd}(A,B,C)$ and $H=F/G$. The triple $A=B+C$ is a solution
of the logarithmic derivative ansatz exactly when the following two conditions hold:
\begin{itemize}
\item $H$ is a polynomial, with simple roots;
\item a root of $G$ is a root of $H$ if and only if it divides one of 
$A,B,C$ in a higher order than others.
\end{itemize}
\end{lemma}
\begin{proof}
A parasitic solution of the straightforward method defines a unique point $(A:B:C)$ in $\PP^2(\CC(x))$,
independent of reduction of common factors. The logarithmic derivatives 
$\varphi'(x)/\varphi(x)$ and $\varphi'(x)/(\varphi(x)-1)$ are the same rational functions
in both the assumed and simplified settings. From (\ref{eq:ldiffabc}) it follows
that $(A:B:C:F)$ must represent the same point in $\PP^3(\CC(x))$ as
the simplified $(A/G:B/G:C/G:H)$.
\end{proof}

\begin{corollary} \label{th:klmparasitic}
Parasitic solutions $(P,Q,R,U,V,W)$ of the logarithmic derivative ansatz $(\ref{logd1})$
for $(k,\ell,m)$-minus-$n$ Belyi functions are characterized as follows.
Let $G=\mbox{gcd}(P^kU,Q^\ell V,R^mW)$ for a parasitic solution.
Let $F$ denote the product of irreducible factors of $U\,V\,W$ each to the power $1$
as in $(\ref{logd1})$, and let $H=PQRF/G$. Then the parasitic solution must have:
\begin{itemize}
\item $H$ is a polynomial, with simple roots;
\item a root of $G$ is a root of $H$ if and only if it divides one of 
$P^kU,Q^\ell V,R^mW$ in a higher order than others.
\end{itemize}
\end{corollary}

\subsection{Using transformations of differential equations} 
\label{sec:pullback}
 
To get an even more restrictive system of algebraic equations, we utilize the fact that
our Belyi functions transform hypergeometric equations to Fuchsian equations
with a small number of singularities. We assume the setting of $(k,\ell,m)$-minus-$n$ functions,
with small $n$. 

The Gauss hypergeometric equation is the Fuchsian equation
\begin{equation} \label{HGE}
\frac{d^2y(z)}{dz^2}+
\left(\frac{C}{z}+\frac{A+B-C+1}{z-1}\right)\,\frac{dy(z)}{dz}+\frac{A\,B}{z\,(z-1)}\,y(z)=0.
%z\,(1-z)\,\frac{d^2y(z)}{dz^2}+
%\left(\C-(\A\!+\!\B\!+\!1)\,z\right)\,\frac{dy(z)}{dz}-\A\,\B\,y(z)=0.
\end{equation}
The singularities are $z=0$, $z=1$, $z=\infty$, and the local exponent differences are
$1-C$, $C-A-B$, $A-B$, respectively. Let $E(e_1,e_2,e_3)$ denote a hypergeometric equation
with the local exponent differences $e_1,e_2,e_3$ assigned to the singular points in some 
order\footnote{There is an orbit of 24 hypergeometric equations with the same local exponent differences 
(acted upon by permutation and change of sign), as reflected by the 24 general solutions by Kummer.}.
Pull-back transformations have the form
\begin{equation} \label{algtransf}
z \longmapsto\varphi(x), \qquad y(z)\longmapsto Y(x)=\theta(x)\,y(\varphi(x)),
\end{equation}
where $\varphi(x)$ is a rational function, and $\theta(x)$ is
a radical function (an algebraic root of a rational function).
Under pull-back transformations, 
the exponent differences are multiplied by the branching order in each fiber.
To get non-singular points above $z\in\{0,1,\infty\}$, some exponent differences have to be restricted
to the value $1/k$, where $k\in\ZZ_{\ge2}$ is a branching order in that fiber \cite{HeunClass}.

A $(k,\ell,m)$-minus-$n$ Belyi covering pulls-back $E(1/k,1/\ell,1/m)$ to a Fuchsian equation
with $n$ singularities, after a proper choice of $\theta(x)$ as described in \cite[\S 5]{HyperbHeun}.
If $n=3$, the pulled-back equation can be normalized to a hypergeometric equation again.
If $n=4$, we can normalize to Heun's equation 
 \begin{equation*} \label{Heun}
\frac{d^2y(x)}{dx^2}+\biggl(\frac{c}{x}+\frac{d}{x-1}+\frac{a+b-c-d+1}{x-t}\biggr)\frac{dy(x)}{dx}
+\frac{abx-q}{x(x-1)(x-t)}y(x)=0.
\end{equation*}
Its singularities are $x=0$, $x=1$, $x=t$, $x=\infty$. The exponent differences there are $1-c$, $1-d$,
$c+d-a-b$, $a-b$, respectively. The {\em accessory} parameter $q$ does not influence the local exponents.
The method uses the following lemma.
\begin{lemma} \label{lm:gpback}
Let $\varphi(x)$ be a Belyi map determined by $(\ref{eq:belyiabc})$.
Hypergeometric equation $(\ref{HGE})$ with
\[
A=\frac12\left(1-\frac1k-\frac1{\ell}-\frac1m\right), \quad
B=\frac12\left(1-\frac1k-\frac1{\ell}+\frac1m\right), \quad
C=1-\frac{1}{\ell}.
\]
is transformed to the following differential equation 
under the pull-back transformation
%\begin{equation*}
$z\mapsto \varphi(x)$, $y(z)\mapsto \left(Q^{m} V\right)^{\!A} \, Y(\varphi(x))$:
%\end{equation*}
\begin{align*}
& \frac{d^2Y(x)}{dx^2}+
\left(\frac{F'}{F}-\frac{U'}{\ell\,U}-\frac{V'}{m\,V}-\frac{W'}{k\,W}\right) \frac{Y(x)}{dx} + \\
& \! + A \left[ B \left( \frac{ h_1h_2\,P^{\ell-2}R^{k-2}\,U\,W}{Q^2 F^2}
- \frac{m^2Q'{}^2}{Q^2} -\frac{V'{}^2}{V^2} \right)+\frac{mQ''}{Q}+\frac{V''}{V}
+\qquad \right. \nonumber \\
& \qquad \left. +\left(\frac1k+\frac1\ell\right)\!\frac{mQ'V'}{Q\,V}
+\left(\frac{m Q'}{Q}+\frac{V'}{V}\right) \!
\left(\frac{F'}{F}-\frac{U'}{\ell\,U}-\frac{V'}{V}-\frac{W'}{k\,W}\right)
\right]Y(x)=0.
\end{align*}
\end{lemma}
\begin{proof} A lengthy symbolic computation, using %(\ref{eq:varphis}) and 
(\ref{logd2}), $\varphi(x)-1=R^kW/Q^\ell V$, etc.
% Elimination of $Q^{\ell-2}R^{m-2}$ gives a homogeneous {\em ``trilinear"} differential equation
% for $P$, $Q$, $R$, of order 2 in $P$ and of order 1 in $Q$ or $R$.
\end{proof}

The transformed equation is to be identified with the target Fuchsian equation
with $n$ singularities. Its local exponents can be conveniently determined 
using Riemann's $P$-symbols. 
Any accessory parameters are new additional variables.
A Fuchsian equation with $n$ singularities has $n-3$ accessory parameters.
The terms to $dY(x)/dx$ are always identical, but comparison of the terms to $Y(x)$
gives new algebraic equations between the undetermined variables 
unless\footnote{Lemma \ref{lm:gpback} fails to give new algebraic relations 
when $1/k+1/\ell+1/m=1$, even if we change the sign of some exponent differences.
The equations with $A\neq0$ are then simple projective or fractional-linear transformations
of the equation with $A=0$.

Another case when Lemma \ref{lm:gpback} does not give anything new
(apart from elimination of new accessory parameters) is when $Q=1$. The underlying benefit of 
this method apparently lies in simplifying the factor $Q$ or $Q^2$ in the denominator of the $Y(x)$ term.
There is also simplification of $V^2$ to $V$ in the same denominator, but that simplification is
apparently implied by the log-diff ansatz, as example in \S \ref{sec:davenport} demonstrates.
} $A=0$. %is $1/k+1/\ell+1/m=1$.

The logarithmic derivative ansatz allowed symbolic elimination of $R$ when $k=2$.
If $k=2$, $\ell=3$ and $m\neq 6$, Lemma \ref{lm:gpback} allows symbolic elimination of $P$
additionally. Elimination of $R,P$ from three differential expressions gives
a non-linear differential equation for $Q$, with the coefficients of (presumably monic) $U,V,W$ 
and accessory parameters as parametric variables. 
After substitution of general polynomial expressions 
% (with undetermined coefficients)
for $Q$ and $U,V,W$, we collect to the powers of $x$ and get a system of algebraic equation
for undetermined coefficients.

Even if $(k,\ell)\neq(2,3)$, most of the coefficients of $P,Q,R$ can be eliminated subsequently
from a sequence of equations obtained by identifying two logarithmic derivatives (\ref{logd2}) 
and the $Y(x)$ terms in Lemma \ref{lm:gpback}. 
If $x=\infty$ is assigned as a point above $\varphi=\infty$,
the highest order terms usually allow elimination of all but $2n-5$ non-homogeneous 
variables\footnote{The number $2n-5$ coincides with the number of parameters 
characterizing the Schwarz map~\cite{Wikipedia}. Like in \cite[\S 6.1]{HyperbHeun},
the pulled-back equation has $n-3$ location and $n-3$ accessory parameters,
plus an undetermined constant multiple. Adding a singularity $(n\mapsto n+1$) 
basically adds two new variables:  the location of the new singularity, and an accessory parameter. } 
(plus one if the affine scalings $x\to\alpha x$ are left to act). 
This is demonstrated by \S \ref{ex:deg54} here and \cite[Example 6.2]{HyperbHeun}, 
where pull-backs to Heun's % or hypergeometric 
equation are reduced to algebraic computations in 3 % or 1 
undetermined values.

The order of $k,\ell,m$ (or the 3 fibers) is not essential, of course. In our algorithm realization,
we sought to assign a point $x=\infty$ to a bachelor point of maximal possible branching order $e$, 
and then assign that fiber as $\varphi=\infty$. With this we take advantage of 
the explicit constants $h_1=h_2=e$ in (\ref{logd2}) and sooner eliminations.
The hardest Gr\"obner basis computation is with the last 3 (or 4 if weighted homogeneous) 
variables. The strategy of using first a total degree, then elimination of 2 variables appears to be fastest
for complicated examples. 
Our implementation in {\sf Maple} 15 can compute all 366 Galois orbits of Belyi functions in \cite{HyperbHeun}.
% in XX minutes...

\subsection{A degree 15 example}
\label{ex:deg54}

Here we demonstrate computation of 
$(2,3,7)$-minus-$4$ Belyi functions with the branching fractions $1/2,1/2,1/2,1/7$, 
of degree $15$.  Let us assign the branching fraction $1/7$ to $x=\infty$. Then $U=V=1$,
and the polynomials $P,Q,R,W$ are monic, without multiple roots, of degree $5,2,6,3$ respectively.
If we would assume $W=x(x-1)(x-t)$, the Heun equation would have \mbox{$a=5/28$},
$b=9/28$ and $c=d=1/2$. To avoid increase of the moduli field, 
we rather assume $W=x^3+w_1x^2+w_2x+w_3$.
% Here the $x^2$ term is zero-ed by a translation $x\to x+\beta$, 
% so that only scaling M\"obius transformations $x\to \alpha\,x$ are left to act.
The transformed Fuchsian equation must have the following term to $Y(x)$: $ab(x-q)/W$.
Rather than fixing $x=0$, we normalize $Q=x^2+c$ by a translation $x\mapsto x+\alpha$. 
The logarithmic derivative ansatz gives
\[
R=3P'Q-7PQ',\qquad P^2=2QR'W+QRW'-7Q'RW,
\]
while Lemma $\ref{lm:gpback}$ gives
\[
\frac{13}{84}\left(\frac{P}{Q^2W}-\frac{49Q'{}^2}{Q^2}\right)
+\frac{7Q''}{Q}+\frac{7Q'W'}{2QW}=\frac{135\,(x-q)}{28W}.
\]
The polynomials $R$, $P$ can be eliminated symbolically from the first and third equations. 
The second equation then expands to
\begin{align*}
& \textstyle 324\left(\frac{29}{12}w_2-\frac{20}3c-\frac{784}{1521}w_1^2+\frac{425}{169}qw_1
-\frac{2025}{676}q^2 \right)x^8 \\ 
& \textstyle +\frac{1350}{13}\left(\frac{193}{5}w_3-\frac{196}{39}w_1w_2+\frac{447}{13}qw_2
-\frac{12307}{195}cw_1-\frac{1239}{13}cq\right)x^7\\
& +\ldots \\
& \textstyle +\frac{1620}{13}c^2\left(\frac34c^2w_2-\frac{405}{52}c^2q^2
-\frac{49}{15}cw_2^2-\frac{784}{45}cw_1w_3-\frac{575}{13}cqw_3-\frac{1372}{39}w_3^2\right)=0.
\end{align*}
The variables $w_2,w_3$ are eliminated by the leading two 
coefficients\footnote{The algebraic system is weighted-homogeneous because of the scaling 
action $x\mapsto \alpha x$, with the weights of $q,c,w_1,w_2,w_3$ equal to $1,2,1,2,3$, 
respectively.}.
{\sf Maple} solves the system immediately. There are $4$ Galois orbits of solutions, $3$ of them 
parasitic\footnote{Two parasitic solutions have $c=w_2=w_3=0$ but different $q/w_1$. 
The other parasitic solution is peculiar: it gives a degree 9 Belyi covering with the branching
pattern $4\,[2]+1=3\,[3]=[7]+1+1$ defined over $\QQ(\sqrt{-7})$. But \cite[Table 4]{HeunClass}
gives Belyi covering $H_{11}$ with this branching pattern defined over $\QQ$. 
The quadratic extension occurs because $x=\infty$ is assigned to a non-unique point
(from the parasitic perspective) with the branching fraction $1/7$.}.
The proper solution is defined over a cubic field $K$. 
% By methods of $\cite[\S 3.3]{BelyiKit}$
A small defining polynomial (see also \S ~ref{sec:nfields}) is found for $K$: 
$\xi^3+2\xi^2+6\xi-8$. The solution is unseemly; it particularly 
has\footnote{Useful arithmetic information about $K$ is given in \S \ref{sec:nfields}.}
\begin{align*}
\frac{c}{w_1^2}=\frac{7^4\,(1-2\xi)\,(3+\xi)\,(1-\xi-\frac12\xi^2)^8}
{5^3\,(1-\xi)^5\,(161-86\xi+112\xi^2)^2}.
\end{align*}
The resulting expression for $\varphi(x)$ is long, but $P$ has a linear factor
over $K$. This can be used to optimize $\varphi$ by affine transformations after
keeping $x=\infty$ as it is, or after assigning it to the $K$-root of $P$.

Using a trick from \S \ref{sec:simplify},
the following expression is obtained after a M\"obius transformation:
\begin{align} \label{eq:deg15}
\varphi(x)=&\frac{1162+4282\xi+1523\xi^2}{27} \,
\frac{(14x+12+\xi-\xi^2)^3}{3x+3-\frac12\xi^2}\times \nonumber\\
& \frac{\left(x^4+x^3-(1+2\xi+\frac12\xi^2)x^2-(5+\xi+\frac32\xi^2)x-4\right)^3}
{\left(5x^2+(13+3\xi-\frac12\xi^2)x+4+4\xi+\xi^2\right)^7}.
\end{align}
This Belyi function is identified as the Galois orbit H45 in \cite{HyperbHeun}.
It is a composition factor of the degree 60 covering H46.

\subsection{Counting parasitic solutions}
\label{sec:parasite}

Computation of the degree 15 covering led to 3 parasitic solutions.
The same number of parasitic solutions was found in \cite[Example 6.2]{HyperbHeun}
when looking for degree 54 Belyi functions with the branching pattern
$27\,[2]=18\,[3]=7\,[7]+2+1+1+1$ by the same combination of the logarithmic derivative ansatz 
and Lemma \ref{lm:gpback}. Here we use Corollary \ref{th:klmparasitic}
to count the number of parasitic solutions for both examples 
of the logarithmic derivative ansatz alone.

We start with the simpler degree 15 example. 
Each single root $x=u$ that is simplified in the vector
$(P^3\!:\!Q^7\!:\!R^2W\!:\!PQRW)$ with some multiplicity is restricted independently.
By a {\em profile} $(\alpha,\beta,\gamma,\delta)$ of the simplified root $x=u$ 
we mean the multiplicities $\alpha\le5$, $\beta\le2$, $\gamma\le6$, $\delta\le3$
with which it divides $P,Q,R,W$, respectively. Lemma \ref{lm:gpback} implies
that the profiles must satisfy one of the following conditions:
\[ \begin{array}{l}
3\alpha=7\beta=2\gamma+\delta=\alpha+\beta+\gamma+\delta;\\
3\alpha=7\beta=\alpha+\beta+\gamma+\delta-1<2\gamma+\delta;\\
3\alpha=2\gamma+\delta=\alpha+\beta+\gamma+\delta-1<7\beta;\\
7\beta=2\gamma+\delta=\alpha+\beta+\gamma+\delta-1<3\alpha.
\end{array}
\]
We cannot have $3\alpha=7\beta$, so only the last two possibilities are left. 
In particular, $\gamma=\alpha+\beta-1$.
For each $\alpha\le5$, $\beta\le2$ we solve linearly for $\gamma,\delta$
and check their non-negativity and upper bounds. We discard 
$\alpha+\beta+\gamma+\delta\le 1$ as giving the non-parasitic solution.
Seven possible profiles are found, with
\[ (\alpha,\beta)\in \{(1,1),(2,1),(3,1),(2,2),(3,2),(4,2),(5,2)\}.
\]
Each profile already gives a branching pattern for a parasitic solution;
it is what is left out of $5\,[3]=7\,[2]+1=6\,[2]+1+1+1$ after the simplification. 
For example, $(\alpha,\beta)=(5,2)$ gives a linear parasitic solution.
Besides, we can combine some profiles to simplifications with two roots:
\[
(\alpha,\beta) \in \{
(1,1)+(1,1);\ (1,1)+(2,1);\ (1,1)+(3,1);\ (2,1)+(3,1) \}.
\]
Further combinations are restricted by $\beta\le 2$ foremost. In total,
we have $7+4$ branching patterns for parasitic solutions. 
Beside the mentioned linear solution, we get the Belyi functions
$H_{11},H_{26},H_{29},H_{32},H_{34},H_{43}$ of \cite[Table 4]{HeunClass}
and the Galois orbits G17, G38, H47 of \cite{HyperbHeun}.
Only $(\alpha,\beta)=(1,1)+(1,1)$ gives a branching pattern ($N_{15}$ in \cite[Table 5]{HeunClass})
with no Belyi functions. The number of parasitic Galois orbits of the logarithmic derivative ansatz 
is 10.

In the setting of degree 54 Belyi functions, we are looking for simplifications
of the polynomial vector $(P^3\!:\!Q^7V\!:\!R^2\!:\!PQRV)$. 
We will conclude that {\em all} Belyi functions of \cite[Tables 2.3.7--2.3.13]{HeunClass}
with a branching fraction $2/7,2/8,\ldots$ or $2/13$ appear as parasitic solutions of the
logarithmic derivative ansatz. Besides, there are 7 parametric Galois orbits, 
most of the Belyi functions in \cite[Table 4]{HeunClass} among the parasitic solutions, 
and many more other parasitic solutions.

Let the profile vector $(\alpha,\beta,\gamma,\delta)$ denote the multiplicities 
$\alpha\le18$, $\beta\le27$, $\gamma\le7$, $\delta\le3$ of a single root $x=u$ dividing
$P,R,Q,V$, respectively. Corollary \ref{th:klmparasitic} leads to the following 
restrictions and profiles:
\begin{itemize}
\item $3\alpha=2\beta=7\gamma+\delta=\alpha+\beta+\gamma+\delta$. We get $18\gamma=5\beta$
and a single profile $P_0:(\alpha,\beta,\gamma,\delta)=(12,18,5,1)$. The root $x=u$ is not coupled 
to other coefficients, hence inclusion of $P_0$ introduces a free parameter in parasitic  solutions. 
Since the degree decreases by 36, there will be at most one parameter in parasitic solutions.
\item $3\alpha=2\beta=\alpha+\beta+\gamma+\delta-1<7\gamma+\delta$. 
We get $\alpha=2(\gamma+\delta-1)$, $\beta=3(\gamma+\delta-1)$ and $\gamma>5\delta-6$. 
We exclude $(\alpha,\beta)=(18,27)$ as not a solution of the weighted-homogeneous system,
and $\gamma+\delta\le 1$. We get 16 profiles. 
\item $3\alpha=7\gamma+\delta=\alpha+\beta+\gamma+\delta-1<2\beta$. 
We get $\delta=3\alpha-7\gamma$, $\beta=6\gamma+1-\alpha$ and $12\gamma+2>5\alpha$. 
Since $\delta\in[0,3]$, $\alpha\in[\frac73\gamma,\frac73\gamma+1]$.
 We get $5$ profiles, with $(\alpha,\gamma)\in\{(5,2),(7,3),(12,5),(14,6),(17,7)\}$. 
 \item $2\beta=7\gamma+\delta=\alpha+\beta+\gamma+\delta-1<3\alpha$. 
We get $\delta=2\beta-7\gamma$, $\alpha=6\gamma+1-\beta$ and $18\gamma+3>5\beta$. 
There are $8$ profiles, with $(\alpha,\gamma)\in\{(3,1),(6,2),(8,3),(11,4),(13,5),(15,6),(16,6),(18,7)\}$. 
\end{itemize} 
There are thus $30$ profiles in total. They can be combined independently to simplification factors with several roots, if only they do not use up the 18, 27, 7, 3 roots of $P,R,Q,V$, respectively.
The profiles $(12,18,6,1)$, $(12,19,5,1)$, $(13,18,5,1)$,  $(12,18,5,2)$ give specializations of
parametric solutions, as they lead to the same simplification of branching patterns as $P_0$.
We will count only combinations of the other 26 profiles.

% In \cite[Definition 2.1]{HyperbHeun}, {\em branching fractions} of a $(k,\ell,m)$-minus-$n$ Belyi 
% functions are defined as quotients of branching orders by the ...

Each profile has a specific action on the branching fractions. For example, $P_0$ removes
one instance of $1/7$ from the starting branching fractions $1/7,1/7,1/7,2/7$. 
The branching fraction $2/7$ at $x=\infty$ is never affected. Further,
\begin{itemize}
\item The profiles with $\delta=1$ replace one instance of $1/7$ by a number from 
$\{2/7,3/7,4/7,5/7,6/7,8/7,1/3,2/3,4/3,1/2\}$. This means that solving for the degree 54 function
by the logarithmic derivative ansatz leads to {\em all} Belyi functions of \cite[Table 2.3.7]{HyperbHeun}
with a branching fraction $2/7$. That is 60 parasitic Galois orbits with 153 dessins.
Parasitic solutions with several branching fractions $2/7$ may be obtained over 
an extension of their moduli fields, as $x=\infty$ is assigned to one of the $2/7$'s. 
Besides, we get several low degree coverings from \cite[Table 4]{HeunClass}.
\item The profiles with $\delta=2$ replace two instances of $1/7$ by a number from 
$\{2/7,3/7,1/3,1/2\}$. This adds a few more coverings from \cite[Table 4]{HeunClass},
and the well-known degree 6 covering $4(x^2-x+1)^3/27x^2(x-1)^2$. 
%Combined, we get 18 out of 38+1 (for $H_{47}$) coverings 
% of \cite[Table 1]{HeunForm} as well, those giving $2\alpha$ or $2\beta$ in the second column.
\item The profiles with $\delta=0$ append a branching fraction $8/7$, $9/7$, $10/7$, $11/7$, $12/7$,
$13/7$, $4/3$, $5/3$, $6/3$, $3/2$ or $4/2$. In this way, Belyi functions that pull-back $E(1/2,1/3,1/7)$ to 
Fuchsian equations with 5, 6 or 7 singularities occur. Some of those Belyi functions pull-back
specific $E(1/2,1/3,1/k)$ to equations with fewer singularities. For example, 
the profile $(12,8,7,0)$ gives the covering E16 of \cite[Table 2.3.13]{HyperbHeun}.
%that pulls-back $E(1/2,1/3,1/13)$. 
Applying the profile $(2,3,2,0)$ three times and adding the profile $(3,4,1,1)$
gives the degree 28 covering G8 of \cite[Table 2.3.8]{HyperbHeun}. 
\end{itemize}
There are no profiles with $\delta=3$. By combining the profiles with different $\delta$ we obtain:
\begin{itemize}
\item All Belyi functions of \cite[Tables 2.3.8--2.3.13]{HeunClass}
with a branching fraction $2/8,2/9,\ldots$ or $2/13$, nearly 50\% of those tables.
Together with the mentioned  functions from \cite[Tables 2.3.7]{HeunClass},
this already gives 116 parasitic Galois orbits (with 263 dessins in total). 
\item Belyi functions B26, C19, D34 of \cite[Tables 2.4.5, 2.4.7]{HeunClass}.
\item All functions in \cite[Table 4]{HeunClass} with a branching order 2 in the last partition
of the third column, except $H_{44}$ but plus $H_{24}, H_{28}, H_{29}, H_{32}$.
This gives 27 out of 48 Belyi functions from \cite[Table 4]{HeunClass}.
\item There are 7 parametric solutions that utilize $P_0$: 
the degree 18 covering in \cite[\S 9]{VidunasFE} is defined actually over $\QQ(\sqrt{-7})$; 
the degree 6 function explicitly given while discussing $\delta=2$;
and $H_2,H_8,H_{32},H_{34},H_{35}$ of \cite[Table 4]{HeunClass}.
\item 235 branching patterns that do not occur in \cite{HeunClass}, \cite{HyperbHeun}.
Expectedly, they give over 200 parasitic Galois orbits.
The whole list of parasitic branching patterns (and known solutions) is given in
\cite[\sf Parasitic54.txt]{HeunURL}.
\end{itemize}
In total, the expected number of parasitic Galois orbits is around 350. 
The use of the implied pull-back from $E(1/2,1/3,1/7)$ to Heun's equation with the exponent differences
$1/7,1/7,1/7,2/7$ in \cite[Example 6.2]{HyperbHeun} leads to just 3 parasitic solutions (with no parameters):
the degree 18 function defined over $\QQ(\sqrt{-7})$, 
and the functions $H_8,H_{34}$ of \cite[Table 4]{HeunClass} of degree 10 or 3.
% \end{example}

\section{The moduli field and obstruction conics}
\label{sec:conic}
When our algorithm finds an explicit Belyi function, there is no a priori reason to assume
that it is optimal in terms of its realization field (field of definition), or, in terms of its bitsize.
Computational tools are needed for both issues.  This section will focus on realization fields,
while \S \ref{sec:simplify} discusses reducing the bitsize by M\"obius transformations
after a realization field has been selected.

Several computational problems arise when a computed Belyi function is not guaranteed to be
expressed over its moduli field. The basic questions are:
\begin{enumerate}
\item Given a Belyi function $\varphi$, how to compute its moduli field $M_{\varphi}$?
% Is there a M\"obius transformation of $\varphi$ in $M_\varphi(x)$?
\item Given $\varphi$, how to determine the fields over which $\varphi$ has a realization
(after a M\"obius transformation)?
\item If $\varphi$ has no realization over $M_\varphi$, can it be realized as a function
on a conic curve defined over $M_\varphi$?
\item If the branching pattern of $\varphi$ has 2 or 3 symmetric fibers, can $\varphi$ be expressed
over a subfield of $M_\varphi$ if the branching fibers are not constrained to $\{0,1,\infty\}$?
\end{enumerate}
Here we recall the relevant definitions and cohomological concepts, and give constructive answers
to the basic questions. Particularly, in \S \ref{SectionC30}--\ref{SectionD45}
we answer the second question by elementary considerations, without direct reference to cohomology.

Let $\OO$ denote the group of M\"obius transformations:
\[ \OO = \left\{ \frac{ax+b}{cx+d} \, | \, a,b,c,d \in \overline{\QQ} {\rm \ with \ }ad-bc \neq 0 \right\} 
\cong {\rm Aut}(\overline{\QQ}(x)/\overline{\QQ}). \]
For $\varphi \in \overline{\QQ}(x)$ let $\OO_{\varphi}$
denote the group of M\"obius automorphisms of $\varphi$:
\[ \OO_{\varphi} = \{ \mu \in O | \varphi \circ \mu = \varphi\} \cong {\rm Aut}(\overline{\QQ}(x)/\overline{\QQ}(\varphi)), \]
Two rational functions $\varphi_1, \varphi_2 \in \overline{\QQ}(x)$ are called 
{\em M\"obius-equivalent}, denoted $\varphi_1 \sim \varphi_2$, 
if there exists $\mu\in\OO$ with $\varphi_1 \circ \mu = \varphi_2$.
Let $\Gamma = {\rm Gal}(\overline{\QQ}/\QQ)$.

Let $\varphi \in \overline{\QQ}(x)$ be a Belyi function. 
A {\em realization field} of $\varphi$ is a number field over which
a M\"obius equivalent function $\varphi\circ\mu$ is defined. 
The {\em moduli field} $M_{\varphi}$ is the fixed field 
of $\{ \sigma \in \Gamma \, | \, \varphi \sim \sigma(\varphi) \}$.
Clearly $M_{\varphi} \subseteq K_{\varphi}$ for any explicit $\varphi\in K_{\varphi}(x)$
over some realization field $K_\varphi$. 
The moduli field is known to be equal to the intersection of the realization fields of $\varphi$.
% all compositions $\varphi\circ\mu$ with $\mu\in O$.

The first question is answered % directly from the definition of the moduli field, 
by checking which Galois conjugates of $\varphi$ are M\"obius-equivalent to $\varphi$,
following the definition of the moduli field. To determine whether $\varphi_1 \sim \varphi_2$, 
we factor the numerator  of $\varphi_1(x) - \varphi_2(y)$.
If it has a factor $p(x,y) \in \overline{\QQ}[x,y]$ with ${\rm deg}_x(p) = {\rm deg}_y(p) = 1$, 
then $\varphi_1 \sim \varphi_2$ and one finds $\mu$ by solving $p(x,y)=0$ with respect to $y$.

The second question is trivial when $M_{\varphi}$ is known to be a realization field.
Otherwise it is canonically answered by Galois cohomology, as elaborated in \S \ref{sec:obstructf}.
The realization fields are determined by a conic curve $C_{\varphi}$ defined over $M_{\varphi}$,
called the {\em obstruction conic}. The realization fields are those extensions $L$ of $M_\varphi$
for which $C_{\varphi}$ has $L$-rational points. First we show how the obstruction conic 
arises directly, in elementary steps, from the above definitions.

\subsection{The obstruction conic for C30}
\label{SectionC30}

By C30 we refer to one of the 366 Galois orbits of minus-4-hyperbolic Belyi functions, see~\cite{HyperbHeun}
for details. We computed this expression, defined over $K=\QQ(\sqrt{-3})$, for its Belyi function:
\begin{equation} \label{eq:c30}
\varphi(x) = \frac{2((x^2+5)\sqrt{-3} - 3x^2 - 60x + 15)(x^2+5x-5)^4}{(12 x)^5}.
\end{equation}
% There is no a priori reason that our algorithm finds the best (in terms of realization field, or in terms of bitsize, see section 4.2),...
Let $\sigma:\sqrt{-3}\mapsto-\sqrt{-3}$ be the non-trivial element of ${\rm Gal}(\QQ(\sqrt{-3})/\QQ)$.
Now $\sigma(\varphi) \sim \varphi$ since we find $\sigma(\varphi)=\varphi\circ\nu$ with $\nu = -5/x$
by factoring the numerator of $\varphi(y)-\sigma(\varphi(x))$. Hence the moduli field is $M=\QQ$.
The symmetry group $\OO_\varphi$ is trivial, since $\varphi(y)-\varphi(x)$ has no linear factors.

Suppose $\varphi \sim g$ for some $g \in L(x)$ with $\sqrt{-3} \not\in L\supset\QQ$.
Write $\varphi = g \circ \mu$ for some $\mu \in \OO$, which must be unique
because $O_{\varphi}$ is trivial. 
That implies $\mu \in L(\sqrt{-3})(x)$ since $\varphi, g \in L(\sqrt{-3})(x)$.
So we can write $\mu = (ax+b)/(cx+d)$ with $a,b,c,d \in L(\sqrt{-3})$ and $ad - bc \neq 0$.

If $c=0$ or $a=0$, we get a contradiction with $\sqrt{-3}\not\in L$.
For instance, if $c=0$ then $g(x)=\varphi(\hat{a}x+\hat{b})$ with $\hat{a},\hat{b}\in L(\sqrt{-3})$. 
The root $-\hat{b}/\hat{a}\in L$ of the denominator of $g(x)$ does not involve $\sqrt{-3}$,
so we may assume $\hat{b}=0$ after an $L(x)$-affine translation. 
But then $\sqrt{-3}$ stays in the numerator.

We may thus assume $c=1$ without loss of generality. 
Write $a = a_0 + a_1 \sqrt{-3}$, $b = b_0 + b_1 \sqrt{-3}$ and $d = d_0 + d_1 \sqrt{-3}$
for some $a_0,a_1,b_0,b_1,d_0,d_1 \in L$.
We can replace $\mu$ by $\mu - a_0$, giving $g(x+a_0)$ still in $L(x)$.
After this we have $a_0 = 0$, but then $a_1\neq 0$.
Replacing $\mu$ by $\mu/a_1$ now gives $g(a_1 x)$ in the place of $g(x)$.
This gives $a_1 = 1$. Therefore we assume $\varphi = g \circ \mu$ with
\begin{equation}  \label{valuem}
	\mu(x) = \frac{ x \sqrt{-3} + b_0 + b_1 \sqrt{-3}}{x + d_0 + d_1 \sqrt{-3}} 
\end{equation}
for some $b_0, b_1,d_0,d_1 \in L$.

Let $\sigma_L:\sqrt{-3} \mapsto -\sqrt{-3}$ be the non-trivial element of ${\rm Gal}(L(\sqrt{-3})/L)$.
Recalling $\nu=-5/x$ in $\sigma(\varphi)=\varphi\circ\nu$, we have
\[ g \circ \sigma_L(\mu) = \sigma_L(g) \circ \sigma_L(\mu) = \sigma_L(g \circ \mu) 
= \sigma(\varphi) = \varphi \circ \nu = g \circ \mu \circ \nu. 
\]
We conclude $\sigma_L(\mu) = \mu \circ \nu$, since the groups $\OO_g, \OO_{\varphi}$ are trivial.
Write the numerator of $\sigma(\mu) - \mu \circ \nu$
as $\sum_{i=0}^2 \sum_{j=0}^1  C_{i,j} \, x^i (\sqrt{-3})^j $ with
$C_{i,j} \in  \QQ[b_0, b_1, d_0, d_1]$.  All $C_{i,j}$ must be zero, giving 
\[ b_1 + d_0 = b_0 - 3 d_1 =  5 + b_0 d_1 - b_1 d_0 = 0. \]
This reduces to
\begin{equation} \label{coniceq}
b_1 = -d_0, \qquad b_0 = 3 d_1, \qquad d_0^2 + 3 d_1^2 + 5 = 0.
\end{equation}
So if $\varphi$ has a realization over $L$ then $d_0^2 + 3 d_1^2 + 5 = 0$
has a solution $d_0,d_1 \in L$.  Conversely, for any solution $d_0, d_1 \in L$, 
one obtains $g = \varphi \circ \mu^{-1}\in L(x)$ with $\mu$ as in (\ref{valuem}), (\ref{coniceq}).
Hence the realization fields for C30 are precisely those number fields that have a rational point
on $d_0^2 + 3 d_1^2 + 5 = 0$.

\begin{remark}
The function $g=\varphi\circ\nu^{-1}$ will be an element of $\QQ(u,v,x)/(u^2+3v^2+5)$.
After a tedious simplification, we obtained this expression for $g$:
\begin{equation} \label{eq:C30uvx}
\frac{\big(\frac13(2u-5)x^2-4vx-2u-5\big)^4
\big( (u-v-10)x^2-(2u+6v)x-3u+3v-30 \big)}{512\,(x^2+3)^5}.
\end{equation}
Any point $(u,v)$ on the conic gives a realization of \AJ{C30}.
Interestingly, any specialization $x\in \QQ$ gives a Belyi function 
% I wouldn't put x \in CC since that can make the denominator 0
in $\QQ(u,v)/(u^2+3v^2+5)$ with the same dessin as \AJ{C30}.
For all $10$ cases of \cite[Table 2]{HyperbHeun} with trivial $\OO_{\varphi}$
we got an expression like (\ref{eq:C30uvx}) that gives conic realizations
with any specialization of $x$.
 
% HOW SURE ARE WE THAT THIS HAPPENS WHENEVER $|\OO_\varphi|=1$?
% WHY WOULD THIS EXPRESSION % (\ref{eq:C30uvx}) 
% IN $x$ NOT HAVE A COMPLICATED ``CONSTANT'' FACTOR IN $\QQ(u,v)/(u^2+3v^2+5)$? 
% CAN WE FORMULATE NICELY THIS OBSERVATION WITHOUT REFERENCE TO DESSINS?
% CAN WE OBTAIN ALL CONIC MODELS OVER A FIELD $L$ BY SUBSTITUTING VALUES $x\in L$?
\end{remark}

\subsection{Obstruction conics generally}
\label{SectionD45}

Given a Belyi function $\varphi(x)\in K(x)$ with a moduli field $M_{\varphi}\subset K$
and trivial $\OO_{\varphi}$, the same routine of expanding $\sigma_L(\mu)=\mu\circ\nu$ 
for an assumed M\"obius transformation $\mu$ and any Galois action $\sigma_L$ 
on $LK \supset L$ produces an equation over $M_\varphi$ obstructing
the realization fields of $\varphi(x)$. Importantly, $\nu$ is unique by $|\OO_\varphi|=1$. 
If the branching pattern of $\varphi$ has a point-couple, assigning those points to $x=0$, $x=\infty$
extends $M_\varphi$ at most quadratically to $M_{\varphi}(\sqrt{A})$ and gives a realization. 
The unique $\nu$ then either $x\mapsto -x$ or has the form $x\mapsto B/x$ for some $B\in M_{\varphi}$.
In the former case, the variable change  $x\mapsto \sqrt{A}\,x$ gives a realization over $M_{\varphi}$.
Otherwise, the following theorem can be applied. 
\begin{theorem} \label{tm:conic}
Suppose that we have a Belyi function $\varphi(x)\in M_\varphi(\sqrt{A})$ where $M_\varphi$ 
is the moduli field, and $A\in M_\varphi$. Suppose that $|\OO_\varphi|=1$. % is trivial.
Let \mbox{$\sigma:\sqrt{A}\mapsto-\sqrt{A}$} be the conjugation in ${\rm Gal}(M_\varphi(\sqrt{A})/M_\varphi)$,
and suppose that $\sigma(\varphi)=\varphi\circ\nu$ with $\nu:x\mapsto B/x$ for some $B\in M_\varphi$.
Then $L\supseteq M_{\varphi}$ is a realization field for $\varphi$ if and only if 
the conic \mbox{$u^2=Av^2+B$} has an $L$-rational point.
\end{theorem}
\begin{proof}
Suppose $\varphi \sim g$ for some $g \in L(x)$ with $\sqrt{A} \not\in L\supset M_{\varphi}$.
Write $\varphi = g \circ \mu$ for some $\mu \in \OO$. By the same arguments as in \S \ref{SectionC30},
we can assume
\begin{equation}  \label{valuemAgain}
	\mu(x) = \frac{ x \sqrt{A} + b_0 + b_1 \sqrt{A}}{x + d_0 + d_1 \sqrt{A}} 
\end{equation}
for some $b_0, b_1,d_0,d_1 \in L$. After expanding the numerator of $\sigma(\mu) - \mu \circ \nu$
we get the equations
\[ b_1 + d_0 = b_0 + A d_1 = b_0 d_1 - b_1 d_0 - B = 0. \]
This reduces to
\begin{equation} \label{coniceqAGAIN}
b_1 = -d_0, \qquad b_0 = -A d_1, \qquad d_0^2  = A  d_1^2 + B.
\end{equation}
The theorem follows.
\end{proof}
This theorem can be used to find the obstruction conics for all cases 
of \cite[Table 2]{HyperbHeun} with trivial $\OO_\varphi$: 
B12, C6, C30, F1, F4, F11, H1, H10, H11, H12.
Those branching patterns have at least two point-couples.
\begin{remark} \label{rm:ratconic}
Theorem $\ref{tm:conic}$ can lead to a conic with a rational point over $M_{\varphi}$.
For example, consider the Galois orbit \AJ{C11} of $\cite{HyperbHeun}$ 
with $M_{\varphi}=\QQ$ and the branching pattern $6+6+1+1=3\,[4]+1+1=7\,[2]$. 
After assigning a point-couple to $x=\infty$, $x=0$,
the following expression over $\QQ(i)$ can be computed:
\begin{equation}
\varphi(x)=\frac{(5x^2+998x+5-12i(x^2-1))\left(x^2-50x+1\right)^6}
% {108x\left(3x^3+39x^2-39x-3+2i(x^3+143x^2+143x+1)\right)^4}.
{108x\left(3(x-1)(x^2+14x+1)+2i(x+1)(x^2+142x+1)\right)^4}.
\end{equation}
The complex conjugation is realized by $x\mapsto 1/x$.
The obstruction conic is then $u^2+v^2=1$, which has obvious $\QQ$-rational points.
Hence \AJ{C11} has realizations over $\QQ$, for example
\begin{equation}
\frac{4\,(4x^2-2x+7)\,(5x^2-2x+8)^6}{27\,(37x^2-16x+64)\,(x^3-5x^2+4x-8)^4}.
\end{equation}
\end{remark}

If $|\OO_\varphi|>1$, the presence of symmetries means that $\varphi$ is a composition
of lower degree rational functions. In particular, we can take $\varphi=\psi\circ\lambda$,
where $\psi$ is the quotient of $\varphi$ by $\OO_\varphi$ (as a covering), 
and the degree of $\lambda$ equals $|\OO_\varphi|$. We can recursively determine
the realization fields of $\psi$.
% as $\OO_\psi$ is trivial. --->  Does not need to be trivial.   E.g. stack a few degree-2 extensions on top of each other.
The realizations fields of $\varphi$ can be decided by comparing 
the realizations of $\varphi$ and $\psi$.

There are 4 Galois orbits with $|\OO_\varphi|>1$ in \cite[Table 2]{HyperbHeun}:
D45, F6, H13, H14. They all have $|\OO_\varphi|=2$. The quadratic quotients by $\OO_\varphi$
are in \cite[Table 2]{HyperbHeun} as well: C30, F4, H12, H10, respectively.
The following two lemmas imply that D45, F6, H13, H14 have the same 
realization fields as their respective quotients by $\OO_\varphi$
(the moduli fields are the same as well). We can say that the obstruction conics
for D45, F6, H13, H14 are those of their quadratic quotients.
\begin{lemma} \label{decomp2} 
Suppose $\varphi \in K(x)$ is a Belyi function with $|\OO_{\varphi}| = 2$.
Then we can write $\varphi = \psi\circ\lambda$ for some $\lambda \in K(x)$ of degree $2$
and $\psi\in K(\lambda)$. If $\varphi$ has a realization over a field $L$, then so does $\psi$.
\end{lemma}
\begin{proof}
Let $\mu$ be the non-identity element of $\OO_{\varphi}$.  Since $\varphi$ is
invariant under ${\rm Gal}(\overline{K}/K)$, the same must be true for $\mu$, and so $\mu\in K(x)$.
Since $\mu$ has order $2$ in $\OO_{\varphi}$, its fixed field $F \subseteq K(x)$ has index 2.
At least one of the functions $x \mu$ or $x + \mu$ is not constant.
Define $\lambda \in F$ as $x \mu$ if $x \mu \not\in K$, and $x + \mu$ otherwise. 
Then $\lambda$ has degree 2, so it generates $F$.
Since $\varphi \in F=K(\lambda)$, 
it follows that $\varphi = \psi\circ\lambda$ for some $\psi \in K(\lambda)$.

If $\varphi$ has a realization $\hat{\varphi} \in L(x)$, then from
the non-identity element of \mbox{$\OO_{\hat{\varphi}}\cong\OO_\varphi$} we can compute 
explicit $\hat{\lambda}$ and $\hat{\psi}\in L(\hat{\lambda})$ in exactly the same way.
As $\overline{L}(\lambda)\cong\overline{L}(\hat{\lambda})$, 
we have $\hat{\lambda}\sim\eta\circ\lambda$ 
for some M\"obius transformation $\eta\in\overline{L}(\lambda)$. 
Then $\hat{\psi}\sim\psi$ by $\psi=\hat{\psi}\circ\eta$, 
thus $\hat{\psi}$ is a realization of $\psi$ over $L$.
\end{proof}

\begin{lemma} \label{decompb} 
Let $K$ be a field of characteristic $0$. 
Suppose $\varphi \in K(x)$ is a Belyi function with $|\OO_{\varphi}|=2$. % and \mbox{$\mbox{\rm char } K=0$}.
Let \mbox{$\varphi=\psi\circ\lambda$} be a decomposition as in the previous lemma. 
Assume that the quadratic covering $\lambda$ branches over a point-couple of $\psi$.
Then $\varphi$ and $\psi$ have the same set of realization fields.
\end{lemma}
\begin{proof}
We only have to prove that if $\psi$ has a realization $\hat{\psi}$ over a field $L$ % (of characteristic $0$)
then $\varphi$ has a realization over $L$. Let $\eta\in\overline{L}(\lambda)$ be the M\"obius transformation
in $\psi=\hat{\psi}\circ\eta$, and let $\hat{\lambda}=\eta\circ\lambda\in\overline{L}(x)$.
Let $P_1,P_2$ be the branching fibers of $\hat{\lambda}$; 
they form a point-couple for $\hat{\psi}$ as images under $\eta$ of the assumed couple for $\psi$. 
The set $\{P_1,P_2\}$ is invariant under ${\rm Gal}(\overline{L}/L)$, because $\hat{\psi}$ is. 
We can construct an explicit $\lambda^*\in L(x)$ of degree 2 that branches above $P_1$ and $P_2$.
This is straightforward if $P_1,P_2 \in L \bigcup \{\infty\}$; otherwise $P_1,P_2$ are the 
roots of an irreducible polynomial $z^2+Az+B \in L[z]$ and we can take
\begin{equation} \label{eq:lambda}
\lambda^*(x)=c\,\frac{(c^2+Ac+B)(x+a)^2-(a-b)^2 B}{(c^2+Ac+B)(x+b)^2-(a-b)^2\,c^2}
\quad\mbox{with any } a,b,c\in L.
\end{equation}
We have $\lambda^*\sim\hat{\lambda}$, hence
$\hat{\psi}\circ\lambda^*\sim\hat{\psi}\circ\hat{\lambda}=\hat{\psi}\circ\eta\circ\lambda=
\psi\circ\lambda=\varphi$, and $\hat{\psi}\circ\lambda^*$ is defined over $L$. 
\end{proof}

\subsection{Galois cohomology obstructions}
\label{sec:obstructf}

For $\varphi \in \overline{\QQ}(x)$, 
let us denote $\Gamma_\varphi=\mbox{Gal}(\overline{\QQ}/M_\varphi)$.
For any $\sigma\in\Gamma_\varphi$
we have $|\OO_\varphi|$ choices for $\mu\in\OO$ in $\sigma(\varphi)=\varphi\circ\mu$.
If for each $\sigma\in\Gamma_\varphi$  we can choose such $\mu_\sigma\in\OO$  
so that $\mu_\sigma\circ\sigma(\mu_\rho)=\mu_{\sigma\rho}$ for any $\sigma,\rho\in\Gamma_\varphi$, 
then we have a cocycle of Galois cohomology \cite{Couveignes94,SerreGC} %\cite{Filimonenkov}
representing an element of $H^1(\Gamma_\varphi, \OO)$. 
The realization fields $L$ are then those
which are mapped to the identity in $H^1({\rm Gal}( \overline{\QQ} / L), \OO)$.
As recalled in \cite{Filimonenkov}, the elements of $H^1( \Gamma_\varphi, \OO)$
are in one-to-one correspondence with isomorphism classes of conic curves over $M_\varphi$.
This is a special case of the construction in \cite[Ch. XIV]{SerreGC}. 

\begin{example} \label{C30D45}
An example of a Galois orbit of Belyi functions without a cocycle is $\AJ{D45}$ in \cite{HyperbHeun}.
The branching pattern is $4\,[5]=4\,[4]+1+1+1+1=10\,[2]$. 
We start with the following realization in $\QQ(\sqrt{-2})$:
\begin{equation*}
%\frac{\big(x^4-10x^2+5+10\sqrt{-2}(x^2+1)\big)^4
%\big(81x^4+1330x^2-255+40\sqrt{-2}(19x^2-21)\big)}{81\,(x^4+18x^2-15)^5}.
\Phi(x)=\frac{\big(x^4-10x^2-5\big)^4\big(243x^4-190x^2+1205+110(19x^2+1)\sqrt{-2}\big)}
{\big(3x^4+18x^2+5-2(9x^2-5)\sqrt{-2}\big)^5}.
\end{equation*}
An obvious symmetry is $x\mapsto -x$. Factorization of $\Phi(x)-\Phi(y)$ shows no other symmetries,
hence $|\OO_\Phi|=2$. The conjugation of $\sqrt{-2}$ is realized by $x\mapsto\pm\sqrt{-5}/x$.
The moduli field is $\QQ$, but no cocycle over $\QQ$ can be formed.
As suggested by Lemmas $\ref{decomp2}$ and $\ref{decompb}$, we can take the quotient of \AJ{D45}
by $\OO_\Phi$. The quotient is the rational function $\Phi(\sqrt{x})$.
It is M\"obius equivalent to the \AJ{C30} function $\varphi(x)$ in $(\ref{eq:c30})$,
as one can check by finding a linear factor of $\varphi(y)-\Phi(\sqrt{x})$. 
As spelled out explicitly by Lemmas $\ref{decomp2}$ and $\ref{decompb}$, 
the realization fields of \AJ{C30} and \AJ{D45} are the same.
\end{example}

A cocycle certainly exists if $|\OO_\varphi|=1$. Other broad case with a cocycle % (perhaps trivial)
is Belyi functions $\varphi$ with a point-couple. As already explained, then we have a realization 
over a quadratic extension of $M_{\varphi}$; the quadratic conjugation is realized by an order 2 
M\"obius transformation that is (importantly) in $M_{\varphi}(x)$. 

Possible $\OO_\varphi$
and existence of cocycles for genus 0 Belyi coverings are classified in Theorem 2 of
\cite[\S 7]{Filimonenkov}. The possible $\OO_\varphi$ form the familiar list of finite subgroups 
of $PSL_2(\CC)$: the cyclic, dihedral, tetrahedral, octahedral and icosahedral groups. 
Cocycles do not exist only if $\OO_\varphi$ is a cyclic group. $M_{\varphi}$ is not a realization 
field only if  $\OO_\varphi$ is a cyclic group or Klein's (dihedral) group with 4 elements. 
There is always a realization over a quadratic extension of $M_{\varphi}$. 

% GENERALIZATION OF $\ref{decomp2}$ and $\ref{decompb}$ ???
% One conic equivalence direction: how does quadratic conjugation interact with a cyclic O_phi/
% Other conic equivalence direction: if ramification is not over a point-couple, how the conic changes?
% Or is it only the moduli field changes?
% If there is a bachelor point, order 2 affine transformation combines with a quadratic conjugation...
% 

% WHAT IS THE "OBSTRUCTION CONIC" IN CASES WITHOUT A COCYCLE?
% HOW WOULD WE RUN THE COMPUTATION OF \S 3.1 ON D45?
% HOW COULD THE MODULI FIELDS AND OBSTRUCTION CONICS 
% FOR OTHER POSSIBLE C30/D45-TYPE PAIRS RELATED? 
% QUOTIENTS BY CYCLIC GROUPS SHOULD NOT BE HARD.

\subsection{Conic models} 

Suppose that the Galois orbit of a Belyi function $\varphi$ has a cocycle 
(or more particularly, $|\OO_\varphi|=1$). If a Galois element $\sigma\in\Gamma$
is represented by a M\"obius transformation $\mu_\sigma\in\OO$, the function $\varphi$
is invariant under the joint action of $\sigma$ and $\mu_\sigma^{-1}$. 
We can find a set of generators of the invariant functions under this action,
and write $\varphi$ in terms of them. The invariant field defines an algebraic curve
over $M_\varphi$ of genus 0, isomorphic (over $M_{\varphi}$) to $\PP^1$ or a conic.

By Theorem 2 in \cite[\S 7]{Filimonenkov}, there is always a realization
over a quadratic extension $M_\varphi(\sqrt{A})$, with $A\in M_\varphi$. 
Let $\mu\in\OO$ be the cocycle representative of those Galois elements 
that conjugate $\sqrt{A}\to-\sqrt{A}$. 
Then $\mu^{-1}=\mu$ and the invariant functions are generated by two non-constant functions among
\[
\frac{x+\mu(x)}{2},\qquad \frac{x-\mu(x)}{2\sqrt{A}}, \qquad x\,\mu(x).
\]
The following special case mimics Theorem \ref{tm:conic}.
% ARE THESE TWO RESULTS EQUIVALENT? 
% CAN WE REMOVE THE CONDITION $|\OO_\varphi|=1$ IN THEOREM \ref{tm:conic},
% AS $x\mapsto B/x$ FOR $B\in M_\varphi$ IS ALWAYS A COCYCLE?
\begin{lemma} \label{lm:conic}
Suppose that we have a Belyi function $\varphi\in M_\varphi(\sqrt{A})$ where $M_\varphi$ 
is the moduli field. Suppose that there is a Galois cocycle that sends the Galois elements that
conjugate \mbox{$\sqrt{A}\to-\sqrt{A}$} to $x\mapsto B/x$ for $B\in M_\varphi$. 
Then $\varphi$ can be written as a function on the conic \mbox{$u^2=Av^2+B$,}
meaning \mbox{$\overline{\QQ}(\varphi)\cong\overline{\QQ}(u,v)/(u^2-Av^2-B)$}.
\end{lemma}
\begin{proof}
The functions
\begin{equation} \label{eq:conicp}
u=\frac12\left(x+\frac{B}x\right), \qquad 
v=\frac1{2\sqrt{A}}\left(x-\frac{B}x\right).
\end{equation}
generate the invariants under the joint Galois and $\mu$ action,
since they determine the orbits $\{u\pm\sqrt{A}v\}$. 
Hence $\varphi\in M_{\varphi}(u,v)$. The generating invariants are related by $u^2=Av^2+B$.
\end{proof}

We say that a genus 0 Belyi function $\varphi$ has a {\em conic model} if
it can be written as a function on the obstruction conic (over $M_\varphi$) with the same dessin.
Most often, conic models offer a compact expression of the Belyi function,
as demonstrated on the examples of C6, C30, F11 in \cite{HyperbHeun}. 
Conic models of C30 can be obtained by specializing $x\in\QQ$ in (\ref{eq:C30uvx}),
by composing the conic model and a parametrization of the conic.

% We say that a Belyi function $\varphi$ has a conic-model if there exist a conic
% $C$ defined over $M_{\varphi}$ such that the function field of $C$ over $M_{\varphi}$ has a
% Belyi function $f: C \rightarrow P^1$ with the same dessin as $\varphi$.

% ADD A COUPLE OF MORE EXAMPLES, LIKE B12, H1?

\begin{remark}
Belyi functions without a cocycle do not have conic models. In particular,
one can compute an expression like $(\ref{eq:C30uvx})$ on $u^2+3v^2+5=0$ for \AJ{D45} 
by composing $(\ref{eq:C30uvx})$ with a specialized version (say, $a=2$, $b=0$, $c=-1$)
of (\ref{eq:lambda}); there $x^2+Ax+B$ is proportional to the long quadratic polynomial
in the numerator of $(\ref{eq:C30uvx})$. But if we specialize $x$ in the obtained expression,
we obtain a dessin for \AJ{C30}, not \AJ{D45}. 
The quadratic covering between \AJ{D45} and \AJ{C30} composes with a 
parametrization of the conic, not with a conic model. 
In fact $\cite{Macrae}$, a conic defined 
over a field $K$ without a $K$-rational point does not have quadratic coverings defined over $K$.
% (in fact, \cite[Section 4]{DessinsFromAGeometricPointOfView} explains why
% no element of $\RR(u)[v]/(u^2+v^2+1)$ can have D45 or F6 as a dessin).
\end{remark}
\begin{remark}
A conic over $M_{\varphi}$ with a point defined over an odd-degree extension
of $M_{\varphi}$ will necessarily have a point over $M_{\varphi}$; see also \cite[\S 8]{Filimonenkov}.
%
% IN \cite[\S 8]{Filimonenkov}, THIS FACT WITH ODD-SIZE GROUPS OF POINTS IS PROVED
% UNDER THE ASSUMPTION OF A COCYCLE. DO WE NEED THAT ASSUMPTION?
%
The example of Remark $\ref{rm:ratconic}$ is bound to have a model
over the moduli field because of the odd size groups $3\,[4]$, $7\,[2]$ in the branching pattern. 
An obstruction can only occur for branching patterns where in each branching index
appears an even number of times in each partition.  
% REALLY?
Accordingly, the entries of \cite[Table~2]{HyperbHeun} have only couples or even-size 
groups of points with the same branching order.
\end{remark}

Conic models for low degree Belyi functions can also be found from scratch.
Consider, for example, the Galois orbit F1 in \cite{HyperbHeun}, with the branching pattern
$4+4=3+3+1+1=3+3+1+1$. 
The function will have the shape $\varphi=u^3L_0$, $1-\varphi=v^3L_1$
on yet to be determined conic in $u,v$ (not necessarily in the canonical form with 3 terms),
where $L_0,L_1$ are linear in $u,v$. The expression $1-u^3L_0-v^3L_1$
would have then a quadratic factor, giving the conic.  Without loss of generality,
we can multiply two quadratic expressions $Q_1Q_2$ with undetermined coefficients.
In the product, the coefficients to $u^2v^2$, $u^2v$, $uv^2$ and to the terms of degree 1, 2 must vanish.
That gives enough restrictions to determine the possibilities. 
One of obtained\footnote{Up to scaling, we get two quartic expressions. The other one is
\[
1-u^3(u-4v)-4v^3(v+2u)=(1+u^2-2uv-2v^2)(1-u^2+2uv+2v^2).
\]
Both factors lead to D7, even if the first one gives a conic with no $\QQ$-rational point.}
factorable quartics is 
\[
1-u^3(u-8v+4)-v^3(v-8u+4).
\]
It factors over $\QQ(\sqrt{-2})$. A conic expression for F1 is $\varphi=u^3(u-8v+4)$,
where $u,v$ are related by one of the quadratic factors. The conic equation could be transformed
to $u^2+3v^2+\sqrt{2}-1=0$ with some work. Conic parametrizations for C30 and F11 could be found
in a similar way because of many pairs of the same branching orders, though then degree 5 or 6 
expressions are assumed to have a quadratic factor. In this way, F11 can be expressed as 
$u^5(1-u-v)/v$ on a complicated conic over $\QQ(\sqrt{5})$.

\subsection{Computations of conics} 

A conic over a number field $K$ can be characterized in several ways: 
in terms of bad primes, skew fields, or Galois cohomology. These descriptions
help to identify isomorphic (over $K$) conics, to compute an isomorphism or
a simpler conic equation.

Birational isomorphism of conics is conveniently decided by a set of {\em bad primes}. 
The bad primes are precisely those for which $\varphi$ has no realization over 
the completion of $M_{\varphi}$ at $\mathfrak{p}$.
The completion at real primes $\mathfrak{p} = \infty$ is isomorphic to $\RR$. 
The number of bad primes is always even. % \cite{ConicPrimes}. 
The obstruction on realization fields can be described by the set of bad primes
without a reference to the conic. 

For example, the conic $u^2 + 3 v^2 + 5=0$ obtained in \S \ref{SectionC30}
for C30 is isomorphic to the conic given by $u^2+2v^2+5=0$ (evident from the realization
$\Psi(\sqrt{x})$ of Example \ref{C30D45}).
Both conics have the same set of bad primes over $\QQ$: 2 and $\infty$. 
A projective isomorphism is $(u:v:1)\mapsto \big(\frac12(u-5):v:\frac12(u+1)\big)$. 

Isomorphism of conics is best computed using the skew fields characterization.
We developed such an implementation \cite[\sf ConicIsom.mpl]{HeunURL} % \cite{URLconic} 
to find birational maps for conics over $\QQ$. 
For the cases like F1, F4, H10-H14 over extensions of $\QQ$, 
we followed the same method  doing case by case computations. 

%MORE DETAILS??

% Finding a birational map between a conic with large coefficients and a birationally equivalent conic with
% smaller coefficients (the ones in Table~\ref{conic}).

\section{Additional algorithms}
\label{OtherAlgorithms}

Beside determination of moduli fields and obstruction conics,
computation of Belyi functions quickly leads to several simplification problems
as obtained first expressions are unruly. And then we wish to compute
the dessins and possible decompositions of computed Belyi maps.

Section 6.3 in \cite{HyperbHeun} gives a list of side problems were encountered in the course
of handling minus-4-hyperbolic Belyi functions. Here we describe our algorithmic solutions
to most of those problems.

% We also used numerous programs that are non-trivial but costed us little time
% since they were already implemented by others.  Notably: Gr\"obner bases, factoring polynomials,
% algorithms for % algorithms for computing with permutation groups, the polredabs command in GP/PARI to find small with permutation groups, the polredabs command in GP/PARI to find small
% polynomials for representing number fields mentioned in item~1 below.

\subsection{Simplification of number fields}
\label{sec:nfields}

Simplification of a definition field $K$ can be done with the {\sf polred} and {\sf polredabs} commands from {\sf GP/PARI}.
It utilizes integral basis and LLL lattice reduction.
But we had to work around a problem.
In all examples we encountered, the integral basis can be computed
without factoring large integers.
It appears that {\sf GP/PARI} is able to avoid unnecessary large integer factorization
in most cases, but it did get stuck on some cases. To cover those,
we developed our own integral basis and polred implementation.
% compute the field discriminant (and consequently, an integral basis)
% by taking a square of large unfactored tail-integers, since (in practice)
% their factors never occur in the field discriminant.

Relatively comfortable realizations of $K$ are obtained by recognizing towers of number field
extensions. The smallest LLL vectors may give non-reduced polynomials
that do not actually define $K$. But instead of discarding them, we note
that those polynomials define subfields of $K$. When $K$ is not a moduli field,
we automatically have the moduli field as a subfield.

Various cubic fields are encountered frequently.
A straightforward simplification of cubic $K$ is obtained by
trying to simplify the radical expression in Cartan's formula.
A root of $X^3-3aX-2b=0$ is $B^{1/3}+a/B^{1/3}$, 
% \[ x=\left(b+\sqrt{b^2-a^3}\right)^{1/3}+\frac{a}{\left(b+\sqrt{b^2-a^3}\right)^{1/3}},
% \]
where $B=b+\sqrt{b^2-a^3}$.
% though the discriminant is $-108(b^2-a^3)$. 
If $\QQ(\sqrt{b^2-a^3})$ is a principal ideal domain,
then $B$ can be factorized using {\sf Maple}'s {\sf numtheory[factorEQ]}.
Taking the cube-free part of the factorization generally leads to smaller $a$, $b$
in the cubic minimal polynomial of the same shape.

This ready factorization in quadratic principal ideal domains is useful
in subsequent simplification of a Belyi function by the scalings $x\mapsto \alpha x$.
If $K$ is not quadratic or a principal ideal domain, but scaling simplification 
of a particular example is desirable, investigation of $K$-primes is necessary.
We recommend to find out which $\QQ$-primes are likely to play a role, 
and factorize the principal $K$-ideals whose norms involve those primes.
Simplification by the units $\alpha\in K$ should not be forgotten either.
For example, the field $\QQ(\xi)/(\xi^3+2\xi^2+6\xi-8)$ of \S \ref{ex:deg54}
has discriminant  $-980$. An integral basis is $1,\xi,\xi^2/2$. 
The units are generated by $1-\xi$.  The class number appears to be 3. 
Here are some principal ideals that factor into ramified $2_R$, $5_R$ and unramified $2_U,5_U$ primes: 
$(\xi)=2_R2_U^2$, $(\xi^2/2)=2_U^3$, $(1-\xi-\xi^2/2)=2_R^3$,  $(1+\xi^2)=5_R5_U^2$,
$(3-4\xi)=5_U^3$, $(1+\xi-\xi^2)=5_R^3$, and of course $(2)=2_R^22_U$, $(5)=5_R^25_U$.
The prime 7 is totally ramified.

\subsection{Simplification of Belyi functions}
\label{sec:simplify}

Given $\varphi \in K(x)$, 
there usually exists a M\"obius-equivalent $\tilde{\varphi}$ of substantially
smaller bitsize than the one that was obtained initially.
We have a collection of algorithms to find such $\tilde{\varphi}$.
We sketch a few. 

To simplify by the scalings $x\mapsto \alpha x$, we consider a polynomials component
$f = a_n x^n + \cdots a_0 x^0$. 
We multiply $x$ by the primes that appear in $a_n, a_0$ to see if that makes $\varphi$ smaller.
This is easy to implement (one prime at a time) when $K = \QQ$, but when $K$ is a number field, we have to multiply prime
ideals.  %  (trying several multiplicities to see which one works best).
If $I$ is the product, then use LLL techniques (similar to {\sf polred}) to find a good element $\alpha \in I$
(use a dot-product where short vectors correspond to $\alpha$'s for which $x\mapsto \alpha x$ is likely to reduce the bitsize).
One can try several products $I$, and for each, try several dot-products (in order to deal with the likely possibility
that multiplying $\alpha$ by a suitable unit reduces the bitsize of $\varphi$).

Scaling is generally most effective at the latest stage. 
For example, we consequently move each bachelor point to $x=\infty$, 
then select a component $f = a_n x^n + \cdots a_0 x^0$, 
clear its second highest term with $x \mapsto x - a_{n-1}/(n a_n)$, and then apply scaling.
If some of the polynomial components have linear factors over $K$,
their $K$-roots can also moved to $x=\infty$.

Suppose $\varphi \in K(x)$ and $S = \varphi^{-1}( \{0,1,\infty\} )$.  
If some $\alpha \in S$ has a minimal polynomial $f \in K[x]$ of degree 3,
we can apply a method of \S \ref{sec:nfields} ({\sf polredabs} or simplification of Cardano's radicals)
to find an optimized polynomial $g$, then compute a M\"obius transformation over $K$
that will send $\alpha$ to a root of $g$, and check if it makes $\varphi$ smaller.
A similar trick works if we have three $\QQ$-points of degree 1,  or one of degree 1 and one of degree 2.

If some $\alpha \in S$ has a minimal polynomial $f \in K[x]$ of degree 4,
we can compute several small polynomials $g$ defining the same quartic field,
and check the $j$-invariants of their roots. If the $j$-invariant for some $g$
coincides with the $j$-invariant of $f$, then a root of $g$ can be sent to a root of $f$
by a M\"obius transformation in $K(x)$. For example, the polynomial $P$ of \S \ref{ex:deg54}
factors into a linear polynomial and a degree 4 polynomial $P'$. 
The $j$-invariant of the 4 roots of $P'$ equals $j(P')=64(19\xi^2+16\xi+32)/3$. 
We eliminate the field generator $\xi$ by computing the resultant
of $P'$ and the field polynomial $\xi^3+2\xi^2+6\xi-8$. The obtained field $L$ has degree 12 over $\QQ$.
We look at the polynomials $g$ corresponding to the smallest LLL-vectors within {\sf polredabs}.
Each $g$ factors over $\QQ(\xi)$, giving a degree 4 factor $F$. 
A M\"obius identification of $P'$ and $F$ is possible only if
the $j$-invariant of the 4 roots of $F$ equals $j(P')$. The simplest degree 12 polynomial
(corresponding to the first vector in the output LLL basis) 
does not lead to the right $j$-invariant. But luckily, the whole LLL basis contain
even two vectors leading to the right $j$-invariant. 
One of them leads to (\ref{eq:deg15}). In general, we would need to search
short LLL-lattice vectors until the right $j$-invariant is found. Possibly, this search is
typically short in our computational context of Belyi functions.

\subsection{Computation of dessins}

Given a Belyi map in $f \in K(x)$ and an embedding $K \rightarrow \CC$,
we wish to compute the dessin d'enfant (in the combinatorial form)
of the image of $f$ under this embedding. The combinatorial data is given
by three permutations $(g_1,g_0,g_{\infty})$ with $g_i \in S_d$, 
satisfying $g_0 g_1 g_{\infty} = 1$ and the generated group 
$<\hspace{-3pt}g_0,g_1\hspace{-3pt}>$ acting transitively on $S_d$.
The permutations give the monodromy action of $f$.
Although algorithms for computing the monodromy 
exist \cite{Verschelde} %\cite{Rybowicz, Verschelde, Deconinck}, 
this was still a considerable amount of work because we had to develop our own implementation,
specifically optimized for rational functions of high degree that ramify over only 3 points.
We used Puiseux series around $x=0$ and $x=1$, and evaluated them at $x=1/2$.  To correctly match the
Puiseux expansions at $x=0$ to those at $x=1$, we compute a large number of terms, but do this at
a finite precision (i.e. floating point).  We preprocess $f$ with a M\"obius transformation
(if the distance from $f(\infty)$ to 0 or 1 is more than 0 but less than $1/2$, then not every expansion converges).
% The performance (in time) depends on the permutation of $x=0$, $x=1$, $x=\infty$.  

To draw the numerous dessins in \cite{HyperbHeun}, we developed a script language
that utilizes (in particular) possible symmetries.
The script commands are run on {\sf Maple}; they are interpreted
via simple {\sf Maple} routines as printing {\sf Latex}'s code for the {\sf picture} environment.

A useful routine is to recognize whether two monodromies $(g_1,g_0,g_{\infty})$ and 
$(\tilde{g}_1,\tilde{g}_0,\tilde{g}_{\infty})$
represent the same dessin. That is, decide whether there is  
$h \in S_d$ such that $h^{-1} g_i h = \tilde{g}_i$ for $i \in \{0,1,\infty\}$.
A solution: suppose that $h(1) = b$ for some yet to be determined $b \in \{1,\ldots,d\}$.
Then $h(g_0^{n_1} g_1^{n_2} \cdots  g_1^{n_k}(1)) = 
\tilde{g}_0^{n_1} \tilde{g}_1^{n_2} \cdots  \tilde{g}_1^{n_k}(b)$ for all $n_1,\ldots,n_k$.
That determines $h$ because $<\hspace{-3pt}g_0,g_1\hspace{-3pt}>$ acts transitively.
So we can find $h$ (if it exists) by checking $d$ cases,  $b=1,2,\ldots,d$.

\subsection{Computation of decompositions}

Decomposition of a Belyi function $\varphi(x)$ into smaller degree rational functions
is decided by the function field lattice between $\CC(x)$ and $\CC(\varphi)$,
as described in \cite[\S 1.7.2]{LandoZvonkin}.    
If $\overline{\QQ}(\varphi) \subseteq L \subseteq \overline{\QQ}(x)$ is a subfield, 
then $L = \overline{\QQ}(g)$ for some $g$ by L\"uroth's theorem, and $\varphi = h(g)$ for some Belyi map $h$. 

The subfield lattice can be computed using the dessins $(g_1,g_0,g_{\infty})$.
 For this, we compute the subgroups $H$ of the monodromy group $G := <\hspace{-3pt}g_0,g_1\hspace{-3pt}>$ 
that contain $\{ g \in G | g(1)=1 \}$.   
Given such $H$,  writing down the action of $g_1$, $g_0$, $g_{\infty}$ on the
cosets of $H$ produces the dessin of the subfield corresponding to $H$.
We then identified the component Belyi maps $h$ % (such as $h$ of item~(4))
(corresponding to the field $L$) by using the full list of classified 
hypergeometric and hypergeometric-to-Heun transformations
in \cite{VidunasFE,HeunForm} and here.
This way we obtained all decompositions of all entries of A/J tables in \cite[Appendix B]{HyperbHeun}.
The detailed decomposition lattices are given in
\cite[\sf Decomposition\_or\_GaloisGroup]{HeunURL}, together with the used notation.
% See also \cite[Section~1.7.2]{LandoZvonkin}. 
%% That's: Graphs on Surfaces and their Applications.  S.K. Lando and A.K Zvonkin

\section{Symbolic application of differential identities}

Here we derive some useful consequences of the logarithmic derivative ansatz 
and Lemma \ref{lm:gpback}. A few known cases are known  \cite[\S 2.5.2]{LandoZvonkin} 
of occurrence of Chebyshev and Jacobi polynomials as parts of Belyi functions.
In \S\S \ref{ex:chebyshev}, \ref{sec:dihedral} we demonstrate how these
cases naturally follow from the methods of % \S \ref{sec:logdiff} and \S \ref{sec:pullback}
\S \ref{ProgramV}, and immediately derive a few similar occurrences of Jacobi polynomials.
Section \ref{sec:davenport} derives an interesting non-linear differential relation for %computing 
Davenport-Stothers triples. 
Because methods of \S \ref{sec:logdiff}, \S \ref{sec:pullback} 
determine requisite polynomials up to a constant multiple, 
we will use the symbol $\cong$ to mean ``equal up to a constant multiple".

\subsection{Chebyshev polynomials} 
\label{ex:chebyshev}

It is well known that Belyi functions with linear dessins d'enfant like
\begin{equation}
\begin{picture}(130,10)(-5,3)
\put(0,6){\circle*3} \put(15,6){\circle3} \put(30,6){\circle*3} \put(45,6){\circle3}
\put(60,6){\circle*3} \put(75,6){\circle3} \put(90,6){\circle*3} \put(105,6){\circle3} 
\put(0,6){\line(1,0){14}} \put(16,6){\line(1,0){28}} \put(46,6){\line(1,0){28}}  \put(76,6){\line(1,0){28}}
\put(120,6){\circle*3} \put(106,6){\line(1,0){14}}
\end{picture}
\end{equation}
are given by the Chebyshev polynomials of the first and second kind:
\begin{equation} \label{eq:chebysh}
T_n(x)=\cos(n\arccos x),\qquad U_n(x)=\frac{\sin(n\arccos x)}{\sin x}.
\end{equation}
This appearance of Chebyshev polynomials is established rather ad hoc.
It is explained by the logarithmic derivative ansatz as follows.
Let the end black points be $x=0$ and $x=1$.
The white points are roots of a degree $n$ monic polynomial $F$, 
and the interior black points are roots of a degree $n-1$ monic polynomial $G$. 
The Belyi function $\varphi(x)$ is a polynomial of even degree $2n$. We have
\begin{equation} \label{eqt:direct}
\varphi(x)=c_0\,F^2, \qquad \varphi(x)-1={c_0}\,x\,(x-1)\,G^2,
\end{equation}
for some constant $c_0$. The logarithmic derivative ansatz gives
\begin{equation}
2n\,G=2\,F',\qquad 2n\,F=(2x-1)\,G+2\,x\,(x-1)\,G'.
\end{equation}
Elimination of $G$ gives the hypergeometric equation for
\begin{equation}
F\cong \hpg21{-n,\,n}{1/2}{\,x}=T_n(1-2x),
\end{equation}
while elimination of $F$ gives the hypergeometric equation for
\begin{equation}
G\cong \hpg21{1-n,\,1+n}{3/2}{\,x} \cong U_{n-1}(1-2x). %=\frac1n\,U_{n-1}(1-2x).
\end{equation}
Up to constant multiples, 
two hypergeometric polynomials have to be identified as $F$, $G$, respectively.
In fact, we have $T_n(x)^2+(1-x^2)U_n(x)^2=1$ by the trigonometric definitions
$(\ref{eq:chebysh})$.

The similar dessin d'enfant
\begin{equation}
\begin{picture}(115,10)(-5,3)
\put(0,6){\circle*3} \put(15,6){\circle3} \put(30,6){\circle*3} \put(45,6){\circle3}
\put(60,6){\circle*3} \put(75,6){\circle3} \put(90,6){\circle*3} \put(105,6){\circle3}
\put(0,6){\line(1,0){14}} \put(16,6){\line(1,0){28}} \put(46,6){\line(1,0){28}}  \put(76,6){\line(1,0){28}} 
\end{picture}
\end{equation}
defines a Belyi function of odd degree $2n+1$ as follows:
\begin{equation} \label{eqt:directAGAIN}
\varphi=c_1\,x\,F^2, \qquad \varphi-1={c_1}\,(x-1)\,G^2,
\end{equation}
for $F,G$ monic polynomials of degree $n$, and some constant $c_1$.
The logarithmic derivative ansatz gives
\begin{equation}
(2n+1)\,G=F+2\,x\,F',\qquad (2n+1)\,F=G+2\,(x-1)\,G'.
\end{equation}
Elimination of $G$ gives the hypergeometric equation for 
\begin{equation}
F\cong \hpg21{-n,\,n+1}{3/2}{x}=\frac{(-1)^n}{n+\frac12}\;\hpg21{-n,\,n+1}{1/2}{1-x}.
\end{equation}
Up to a constant multiple, % $(3/2)_n/n!=(2n+1)!/4^n(n!)^2$ 
this is the Jacobi polynomial $P_n^{(1/2,-1/2)}(1-2x)$. 
In $\cite[pg. 243]{specfaar}$, these polynomials are identified as
\begin{equation}
V_n(x)\cong \frac{\sin\,(n+\frac12)\arccos x}{\sin\frac12\arccos x},
\end{equation}
and called Chebyshev polynomials of the third kind.

\subsection{Davenport-Stothers triples} 
\label{sec:davenport}

An interesting arithmetic problem is to find large co-prime integers $f,g$ such that 
the difference $f^3-g^2$ is small \cite{hall71}. An analogous question for polynomials in $\CC[x]$ is:
given a polynomial $F$ of degree $2n$ and a co-prime polynomial $G$ of degree $3n$,
how small can the degree of $H=F^3-G^2$ be? The answer is $n+1$, as proved 
by Davenport \cite{davenport65} and Stothers \cite{stothers81}. %\cite{mason84}  
The minimal value is achieved exactly when $\varphi=F^3/H$ is a Belyi function.
These results can be proved by applying the Hurwitz formula 
to the genus 0 covering $\varphi$.

The triples $(F,G,H)$ with the sharp $\deg H=n+1$ 
are called {\em Davenport-Stothers triples}. %\cite{shioda08},  \cite[\S 2.5.1]{LandoZvonkin}
The point $x=\infty$ has then the branching order $5n-1$. 
The logarithmic derivative ansatz gives the relations
\begin{eqnarray*}
(5n-1)G=3F'H-FH',\qquad (5n-1)F^2=2G'H-GH'.
\end{eqnarray*}
Elimination of $G$ gives
\begin{equation} \label{eq:dselimr}
(5n-1)^2F^2=6F''H^2+F'H'H-2FH''H+FH'^2.
\end{equation}
This formula can be rewritten as
\begin{equation} \label{eq:uqu}
\frac{(5n-1)^2\,F-H'^2}{H}=\frac{6F''H+F'H'-2FH''}{F}.
\end{equation}
Since $F$, $H$ are co-prime, this rational function must be a polynomial.
Let $Z$ denote this polynomial. It has degree $n-1$, 
and the leading coefficient is equal to $12n(2n-1)$.
We have two expressions for $Z$ in (\ref{eq:uqu}).
One of them implies $F\cong H\,Z+H'^2$,
the other is homogeneous in $F$ and its derivatives.
Elimination of $F$ gives %\footnote{Alternatively, the eliminations of $G$, $F$ can be substituted
% into $F^3-G^2=cH$, with the constant $c$ because we assume $H$ is monic. The consequence is that
% \begin{align*} -H\left(6H'H''+2H'Z+3HZ'\right)^2+(12H''+7Z)H'{}^4 
% +6HH'{}^3Z'+3HH'{}^2Z^2+H^2Z^3  
% \end{align*} is equal to the constant $(5n-1)^6c$.
% Differentiation of this expression gives the product of  (\ref{eq:dshz}) with
% $6HH'H''-H'{}^3+2HH'Z+3H^2Z'.$
% In a sense, the long expression is an ``integral" of (\ref{eq:dshz}).
% The extra factor (with 4 terms) represents parasitic solutions only.}
an equation independent of $n$:
\begin{equation} \label{eq:dshz}
H'H'''+H''^2+\frac{H''Z}3+\frac{13}{12}\,H'Z'+\frac{HZ''}{2}=\frac{Z^2}{12}.
\end{equation}

Instead of looking for the polynomials $F$, $G$ of degree $2n$, $3n$, 
we could look for polynomials $H$, $Z$ of degree $n\pm1$ satisfying 
(\ref{eq:dshz}) % or ( \ref{eq:dshz2}).
For comparison, extensive computations in \cite{ElkiesWatkins}
reduce the problem of finding Davenport-Stothers triples to looking for polynomials
$A,B,C$ with $F=A^2+B$, $G=A^3+3AB/2+C$, $(\deg A,\deg B,\deg C)=(n,n-1,n-2)$,
$\deg (3B^2-8AC)=n-3$, etc. 
It would be useful
%\footnote{The pull-back transformation of \S \ref{sec:pullback} 
%not give anything new. The starting equation is $\hpgde{1/2,1/3,\alpha}$.
%After the pull-back, the local exponents at the roots of $U$ are 0 and $\alpha$,
%and at $x=\infty$ are $n/2-3n\alpha$ and $n/2+(2n-1)\alpha$.
%The pulled-back equation is
%\[ %begin{equation}
%\frac{d^2Y(x)}{dx^2}+\frac{(1-\alpha)H'}{H}\frac{dY(x)}{dx}
%+\frac{1-6\alpha}{12}\left(\frac{(1+6\alpha)\,Z}{12\,H}+\frac{H''}{H}\right)Y(x)=0.
%\] %end{equation}
%}
to get differential relations for $A,B,C$.

\subsection{Cyclic monodromy and Jacobi polynomials}
\label{sec:dihedral}

Jacobi polynomials \cite{specfaar} are classical 
orthogonal polynomials\footnote{Orthogonality
of Jacobi polynomials $P_{x}^{(\alpha,\beta)}(x)$ is properly defined only when 
$\alpha>-1$, $\beta>-1$. In the cases considered here, these inequalities are
routinely not satisfied. Therefore orthogonality considerations do not apply here.
But usefully, the considered Jacobi polynomials have zeroes outside the real line.} 
on $[-1,1]\subset\RR$.  They are defined by the hypergeometric expression
\begin{equation} \label{eq:jacobi}
P_n^{(\alpha,\beta)}(x)=\frac{(1+\alpha)_n}{n!}\,
\hpg{2}{1}{-n,\,n+1+\alpha+\beta}{1+\alpha}{\,\frac{1-x}2\,}.
\end{equation}
By adjusting the two parameters, any $\hpgo21$ polynomial can be considered as
a Jacobi polynomial. In particular, transformations \cite[\S 4]{degeneratehpg} 
of hypergeometric polynomials imply
\begin{align*}
P^{(\alpha,\beta)}_n(1-2x) =
(-1)^n\,P^{(\beta,\alpha)}_n(2x-1) 
= (1-x)^n\,P^{(\alpha,-2n-1-\alpha-\beta)}_n\!\left(\frac{1+x}{1-x}\right).
\end{align*}
Incidentally, $\alpha,\beta,-2n-1-\alpha-\beta$ are the local exponent differences 
(at $x=0$, $x=1$, $x=\infty$, respectively) of
the hypergeometric equation for $P^{(\alpha,\beta)}_n(1-2x)$. 

Consider the {\em double flower} dessin in Figure \ref{fg:jacobi}~\refpart{a},
with any number $k\ge 0$, $\ell\ge 0$ of petals at the ends, and any number $N$
of intervals on the stalk. It was observed by Magot \cite{Magot97} (see also \cite[\S 2.5.2]{LandoZvonkin})
that the Belyi function $\varphi(x)$ for this dessin d'enfant is expressed in terms of Jacobi polynomials.
If we put the blossoms at $x=0$ and $x=\infty$,
then $\varphi(x)=x^{2k+1}\,\Theta_2(x)^2/\Theta_1(x)^2$ with 
\begin{align*}
\Theta_1(x)&=(1-x)^{N+k+\ell}\,P_{N+k+\ell}^{(-k-1/2,-\ell-1/2)}\!\left(\frac{1+x}{1-x}\right), \\
% =P_{(n+k+\ell)/2}^{(-k-1/2,-n)}(1-2x),\\
\Theta_2(x)&= (1-x)^{N-1}\,P_{N-1}^{(k+1/2,\,\ell+1/2)}\!\left(\frac{1+x}{1-x}\right).
% = P_{(n-k-\ell)/2-1}^{(k+1/2,-n)}(1-2x).
\end{align*}
The hypergeometric equation for $\Theta_1(x)$ and $x^{k+1/2}\,\Theta_2(x)$ is
$E(k+1/2,\ell+1/2,n)$, where $n=2N+k+\ell$ so that $N=(n-k-\ell)/2$.  
The point $x=1$ is a branching point of order $n$. We could also write
\begin{align*}
\Theta_1(x)=P_{N+k+\ell}^{(-k-1/2,-n)}(1-2x),\qquad
\Theta_2(x)= P_{N-1}^{(k+1/2,-n)}(1-2x).
\end{align*}
The occurrence of Jacobi polynomials can be explained as follows.
The branching pattern of the double flower dessin implies that $\varphi$ 
transforms $E(1/2,1/2,1)$ to $E(k+1/2,\ell+1/2,n)$. The monodromy of both 
hypergeometric equations is $\cong\ZZ/2\ZZ$. The pull-back covering $\varphi$ 
is actually $s_0^{-1}\circ s_1$, where $s_0,s_1$ are corresponding 
Schwarz maps \cite{Wikipedia}  for $E(1/2,1/2,1)$, $E(k+1/2,\ell+1/2,n)$, respectively. 
We can take $s_0=\sqrt{x}$,
then $s_1$ is (up to a constant multiple) a quotient of two hypergeometric solutions 
of $E(k+1/2,\ell+1/2,n)$. The hypergeometric solutions can be written as Jacobi polynomials,
and $\varphi=s_1^2$. The degree of the $\varphi$ equals $n+k+\ell$.

A pull-back from $E(1/2,1/2,1)$ to $E(k+1/2,\ell+1/2,n)$ with odd $n+k+\ell$ can be considered as well.
Then we have the same expression $\varphi(x)=x^{2k+1}\,\Theta_2(x)^2/\Theta_1(x)^2$ with
\begin{align*}
\Theta_1(x)&\cong P_{N'+k}^{(-k-1/2,-n)}(1-2x)
= (1-x)^{N'+k}\,P_{N'+k}^{(-k-1/2,\ell+1/2)}\!\left(\frac{1+x}{1-x}\right),\\
\Theta_2(x)&\cong P_{N'+\ell}^{(k+1/2,-n)}(1-2x)
= (1-x)^{N'+\ell}\,P_{N'+\ell}^{(k+1/2,-\ell-1/2)}\!\left(\frac{1+x}{1-x}\right).
\end{align*}
Here $N'=(n-k-\ell-1)/2$. If $N'\ge 0$, the dessin is 
depicted\footnote{We use the same dessin plotting convention as in \cite{HyperbHeun}.
White points of order 2 are not depicted, but the edge going through them is drawn thick.}
in Figure \ref{fg:jacobi}~\refpart{b}. But a dessin is possible for $N'<0$ as well,
as depicted in Figure \ref{fg:jacobi}~\refpart{c} with $M'=1-2N'$, $K=k+N'$, $L=\ell+N'$.
The positive integers $k,\ell,n$ satisfy the triangle inequalities $n<k+\ell$, $k<\ell+n$,
$\ell<k+n$, and $M'$ is odd then. Figure \ref{fg:jacobi}~\refpart{c} 
is valid with $M'\in 2\ZZ$ as well, but then the % respective
pull-back is to a hypergeometric equation with trivial monodromy.

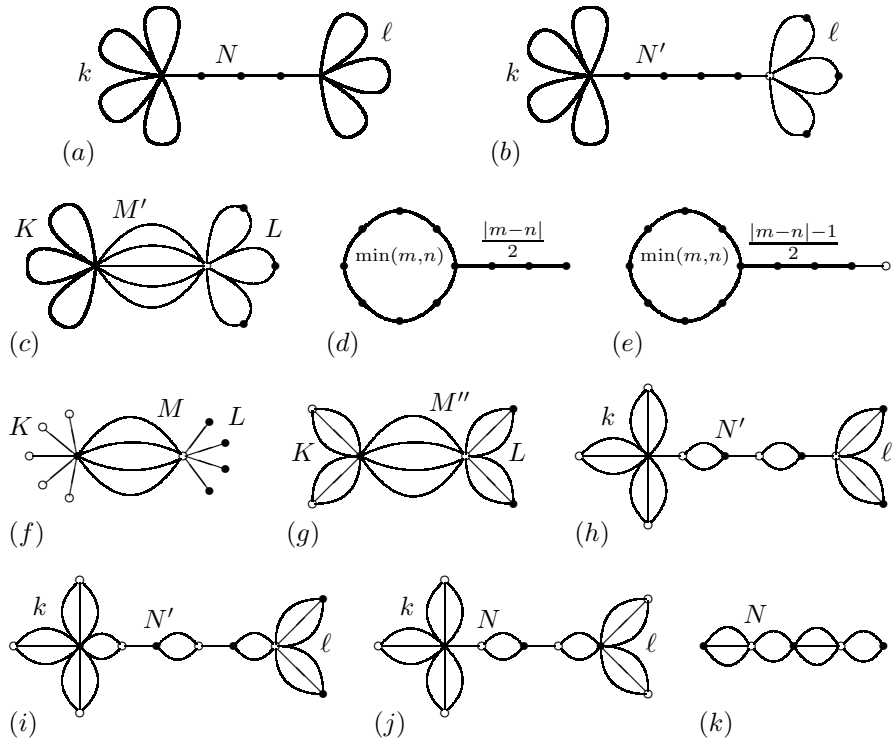
\begin{figure}\begin{picture}(320,280) \thicklines
\put(24,220){$(a)$} \qbezier(62,252)(75,278)(62,278) \qbezier(62,252)(49,278)(62,278) \qbezier(62,252)(75,226)(62,226) \qbezier(62,252)(49,226)(62,226) \qbezier(62,252)(47,277)(40,266) \qbezier(62,252)(33,255)(40,266) \qbezier(62,252)(47,227)(40,238) \qbezier(62,252)(33,249)(40,238) \put(62,252){\line(1,0){60}} \qbezier(122,252)(148,265)(148,252) \qbezier(122,252)(148,239)(148,252) \qbezier(122,252)(147,267)(136,274) \qbezier(122,252)(125,281)(136,274) \qbezier(122,252)(147,237)(136,230) \qbezier(122,252)(125,223)(136,230) \put(77,252){\circle*3} \put(92,252){\circle*3} \put(107,252){\circle*3} 
\put(30,250){$k$} 
\put(145,265){$\ell$} 
\put(82,257){$N$} 
\put(186,220){$(b)$} \qbezier(224,252)(237,278)(224,278) \qbezier(224,252)(211,278)(224,278) \qbezier(224,252)(237,226)(224,226) \qbezier(224,252)(211,226)(224,226) \qbezier(224,252)(209,277)(202,266) \qbezier(224,252)(195,255)(202,266) \qbezier(224,252)(209,227)(202,238) \qbezier(224,252)(195,249)(202,238) \put(224,252){\line(1,0){56}} \put(280,252){\circle*3} {\thinlines\put(280,252){\line(1,0){11}} \put(292,252){\circle3} \qbezier(293,253)(315,265)(318,252) \qbezier(293,251)(315,239)(318,252) \put(318,252){\circle*3} \qbezier(293,253)(315,264)(306,274) \qbezier(293,251)(315,240)(306,230) \put(306,274){\circle*3} \put(306,230){\circle*3} \qbezier(306,274)(293,278)(292,253) \qbezier(306,230)(293,226)(292,251) }\put(238,252){\circle*3} \put(252,252){\circle*3} \put(266,252){\circle*3} 
\put(192,250){$k$} 
\put(314,265){$\ell$} 
\put(241,257){$N'$} 
\put(4,148){$(c )$} \put(37,180){\circle*3} \qbezier(37,180)(11,193)(11,180) \qbezier(37,180)(11,167)(11,180) \qbezier(37,180)(12,195)(23,202) \qbezier(37,180)(34,209)(23,202) \qbezier(37,180)(12,165)(23,158) \qbezier(37,180)(34,151)(23,158) {\thinlines\put(37,180){\line(1,0){41}} \put(79,180){\circle3} \qbezier(37,180)(58,195)(78,181) \qbezier(37,180)(58,165)(78,179) \qbezier(37,180)(58,211)(78,181) \qbezier(37,180)(58,149)(78,179) \qbezier(80,181)(102,193)(105,180) \qbezier(80,179)(102,167)(105,180) \put(105,180){\circle*3} \qbezier(80,181)(102,192)(93,202) \qbezier(80,179)(102,168)(93,158) \put(93,202){\circle*3} \put(93,158){\circle*3} \qbezier(93,202)(80,206)(79,181) \qbezier(93,158)(80,154)(79,179) }
\put(6,191){$K$} 
\put(101,191){$L$} 
\put(43,198){$M'$} 
\put(124,148){$(d)$} \put(138,194){\circle*3} \put(166,166){\circle*3} \qbezier(138,194)(152,208)(166,194) \qbezier(166,166)(152,152)(138,166) \put(166,194){\circle*3} \put(138,166){\circle*3} \qbezier(166,194)(180,180)(166,166) \qbezier(138,166)(124,180)(138,194) \put(152,201){\circle*3} \put(152,159){\circle*3} \put(173,180){\circle*3} \put(131,180){\circle*3} \put(173,180){\line(1,0){42}} \put(215,180){\circle*3} \put(187,180){\circle*3} \put(201,180){\circle*3} 
\put(135,179){$\stackrel{\min(m,n)}{}$} 
\put(182,187){$\frac{|m-n|}2$} 
\put(232,148){$(e)$} \put(246,194){\circle*3} \put(274,166){\circle*3} \qbezier(246,194)(260,208)(274,194) \qbezier(274,166)(260,152)(246,166) \put(274,194){\circle*3} \put(246,166){\circle*3} \qbezier(274,194)(288,180)(274,166) \qbezier(246,166)(232,180)(246,194) \put(260,201){\circle*3} \put(260,159){\circle*3} \put(281,180){\circle*3} \put(239,180){\circle*3} \put(281,180){\line(1,0){42}} \put(323,180){\circle*3} {\thinlines\put(323,180){\line(1,0){12}} \put(336,180){\circle3} }\put(295,180){\circle*3} \put(309,180){\circle*3} 
\put(243,179){$\stackrel{\min(m,n)}{}$} 
\put(283,186){$\frac{|m-n|-1}2$} {\thinlines
\put(4,76){$(f)$} \put(30,108){\circle*3} \put(30,108){\line(-1,0){17}} \put(12,108){\circle3} \put(30,108){\line(-6,-5){12}} \put(30,108){\line(-6,5){12}} \put(17,119){\circle3} \put(17,97){\circle3} \put(30,108){\line(-1,-5){3}} \put(30,108){\line(-1,5){3}} \put(27,124){\circle3} \put(27,92){\circle3} \qbezier(30,108)(50,118)(69,109) \qbezier(30,108)(50,98)(69,107) \qbezier(30,108)(50,137)(69,109) \qbezier(30,108)(50,79)(69,107) \put(70,108){\circle3} \put(71,108){\line(3,-1){15}} \put(71,108){\line(3,1){15}} \put(86,113){\circle*3} \put(86,103){\circle*3} \put(71,107){\line(3,-4){9}} \put(71,109){\line(3,4){9}} \put(80,121){\circle*3} \put(80,95){\circle*3} 
\put(4,116){$K$} 
\put(87,120){$L$} 
\put(60,123){$M$} 
\put(109,76){$(g)$} \put(137,108){\circle*3} \put(137,108){\line(-1,-1){17}} \put(137,108){\line(-1,1){17}} \put(119,126){\circle3} \put(119,90){\circle3} \qbezier(137,108)(119,108)(119,125) \qbezier(137,108)(119,108)(119,91) \qbezier(137,108)(137,126)(120,126) \qbezier(137,108)(137,90)(120,90) \qbezier(137,108)(157,118)(176,109) \qbezier(137,108)(157,98)(176,107) \qbezier(137,108)(157,138)(176,109) \qbezier(137,108)(157,78)(176,107) \put(177,108){\circle3} \put(178,107){\line(1,-1){17}} \put(178,109){\line(1,1){17}} \put(195,126){\circle*3} \put(195,90){\circle*3} \qbezier(178,108)(195,108)(195,126) \qbezier(178,108)(195,108)(195,90) \qbezier(177,109)(177,126)(195,126) \qbezier(177,107)(177,90)(195,90) 
\put(111,106){$K$} 
\put(193,106){$L$} 
\put(163,125){$M''$} 
\put(218,76){$(h)$} \put(246,108){\circle*3} \put(246,108){\line(-1,0){25}} \put(220,108){\circle3} \qbezier(246,108)(233,121)(221,109) \qbezier(246,108)(233,95)(221,107) \put(246,108){\line(0,1){25}} \put(246,108){\line(0,-1){25}} \put(246,82){\circle3} \put(246,134){\circle3} \qbezier(246,108)(233,121)(245,133) \qbezier(246,108)(233,95)(245,83) \qbezier(246,108)(259,121)(247,133) \qbezier(246,108)(259,95)(247,83) \put(246,108){\line(1,0){12}} \put(259,108){\circle3} \qbezier(260,109)(267,117)(274,109) \qbezier(260,107)(267,99)(274,107) \put(275,108){\circle*3} \put(275,108){\line(1,0){12}} \put(288,108){\circle3} \qbezier(289,109)(296,117)(303,109) \qbezier(289,107)(296,99)(303,107) \put(304,108){\circle*3} \put(304,108){\line(1,0){12}} \put(317,108){\circle3} \put(318,107){\line(1,-1){17}} \put(318,109){\line(1,1){17}} \put(335,126){\circle*3} \put(335,90){\circle*3} \qbezier(318,108)(335,108)(335,126) \qbezier(318,108)(335,108)(335,90) \qbezier(317,109)(317,126)(335,126) \qbezier(317,107)(317,90)(335,90) 
\put(228,120){$k$} 
\put(334,106){$\ell$} 
\put(271,115){$N'$} 
\put(4,4){$(i)$} \put(31,36){\circle*3} \put(31,36){\line(-1,0){24}} \put(6,36){\circle3} \qbezier(31,36)(18,49)(7,37) \qbezier(31,36)(18,23)(7,35) \put(31,36){\line(0,1){24}} \put(31,36){\line(0,-1){24}} \put(31,11){\circle3} \put(31,61){\circle3} \qbezier(31,36)(18,49)(30,60) \qbezier(31,36)(18,23)(30,12) \qbezier(31,36)(44,49)(32,60) \qbezier(31,36)(44,23)(32,12) \qbezier(31,36)(39,45)(46,37) \qbezier(31,36)(39,27)(46,35) \put(47,36){\circle3} \put(48,36){\line(1,0){12}} \put(60,36){\circle*3} \qbezier(61,37)(68,45)(75,37) \qbezier(61,35)(68,27)(75,35) \put(76,36){\circle3} \put(77,36){\line(1,0){12}} \put(89,36){\circle*3} \qbezier(90,37)(97,45)(104,37) \qbezier(90,35)(97,27)(104,35) \put(105,36){\circle3} \put(106,35){\line(1,-1){17}} \put(106,37){\line(1,1){17}} \put(123,54){\circle*3} \put(123,18){\circle*3} \qbezier(106,36)(123,36)(123,54) \qbezier(106,36)(123,36)(123,18) \qbezier(105,37)(105,54)(123,54) \qbezier(105,35)(105,18)(123,18) 
\put(13,48){$k$} 
\put(122,34){$\ell$} 
\put(55,43){$N'$} 
\put(142,4){$(j)$} \put(169,36){\circle*3} \put(169,36){\line(-1,0){24}} \put(144,36){\circle3} \qbezier(169,36)(156,49)(145,37) \qbezier(169,36)(156,23)(145,35) \put(169,36){\line(0,1){24}} \put(169,36){\line(0,-1){24}} \put(169,11){\circle3} \put(169,61){\circle3} \qbezier(169,36)(156,49)(168,60) \qbezier(169,36)(156,23)(168,12) \qbezier(169,36)(182,49)(170,60) \qbezier(169,36)(182,23)(170,12) \put(169,36){\line(1,0){13}} \put(183,36){\circle3} \qbezier(184,37)(191,45)(198,37) \qbezier(184,35)(191,27)(198,35) \put(199,36){\circle*3} \put(199,36){\line(1,0){12}} \put(212,36){\circle3} \qbezier(213,37)(220,45)(227,37) \qbezier(213,35)(220,27)(227,35) \put(228,36){\circle*3} \put(228,36){\line(1,-1){17}} \put(228,36){\line(1,1){17}} \put(246,54){\circle3} \put(246,18){\circle3} \qbezier(228,36)(246,36)(246,53) \qbezier(228,36)(246,36)(246,19) \qbezier(228,36)(228,54)(245,54) \qbezier(228,36)(228,18)(245,18) 
\put(152,48){$k$} 
\put(245,34){$\ell$} 
\put(181,43){$N$} 
\put(265,4){$(k)$} \put(267,36){\circle*3} \qbezier(267,36)(276,50)(284,37) \qbezier(267,36)(276,22)(284,35) \put(267,36){\line(1,0){17}} \put(285,36){\circle3} \qbezier(286,37)(294,47)(301,36) \qbezier(286,35)(294,25)(301,36) \put(301,36){\circle*3} \put(301,36){\line(1,0){17}} \qbezier(301,36)(310,50)(318,36) \qbezier(301,36)(310,22)(318,36) \put(319,36){\circle3} \qbezier(320,37)(328,47)(335,36) \qbezier(320,35)(328,25)(335,36) \put(335,36){\circle*3} 
\put(282,45){$N$} }
\end{picture}\caption{Dessins d'enfant for Jacobi polynomials}\label{fg:jacobi}\end{figure}

\begin{remark}
The equation $E(k+1/2,\ell+1/2,n)$ with $k,\ell,n\in\ZZ$ has either logarithmic solutions
or the $\ZZ/2\ZZ$ monodromy. The distinction appears to be tricky \cite{degeneratehpg}.
The dessins in Figure \ref{fg:jacobi}~\refpart{a}--\refpart{c} illustrate the distinction nicely.
The pull-back Belyi covering is possible exactly when the monodromy is $\ZZ/2\ZZ$.
If $k+\ell+n$ is even, this is the case only when $n>k+\ell$. 
If $k+\ell+n$ is odd, we should have either $n>k+\ell$ or the three triangle inequalities satisfied.
\end{remark}

\begin{remark}
The monodromy $\ZZ/2\ZZ$ can be interpreted as a dihedral monodromy.
Hence pull-back computations in $\cite{DihedralTR}$ can be applied. In fact, 
the logarithmic derivative ansatz and Lemma $\ref{lm:gpback}$ have been 
basically used in $\cite[\S 5.3]{DihedralTR}$ with $k=\ell=2$.
Thereby pull-back transformations from $\hpgde{1/2,1/2,1/m}$ to
$\hpgde{k+1/2,\ell+1/2,n/m}$ are obtained\footnote{with any $k,\ell$, 
as we are now outside of Lemma $\ref{lm:gpback}$.}
of degree \mbox{$d=(k+\ell)m+n$}.
The pull-back coverings are Belyi functions defined by the polynomial identity 
\begin{equation} \label{eq:dihedid}
\Theta_1(x)^2-x^{2k+1}\Theta_2(x)^2=(1-x)^n\Psi(x)^m.
\end{equation}
It is proved in $\cite[\S 5.3]{DihedralTR}$ that:
\begin{itemize}
\item $\Psi(x)$ is a solution of a third order Fuchsian
equation\footnote{As predictable by differential Galois theory \cite{kleinvhw}, % singulm1},
the linear Fuchsian equation is the second symmetric tensor power %equation 
of $\hpgde{k+1/2,\ell+1/2,n/m}$. }
% Equation (\ref{eq:q23m}) with $m=2$ is the third symmetric tensor power of 
% a dihedral hypergeometric equation (with the monodromy of 6 elements). 
% The polynomial $Q$ has to be identified as $\Theta_1(x)$ while setting $m=3$ here.}
with the singularities at $x=0,1,\infty$. 
\item $\Theta_1(x)$, $x^{k+1/2}\,\Theta_2(x)$ are solutions of a second order Fuchsian
equation with the singularities at $x=0,1,\infty$ and at the roots of $\Psi(x)$.
\end{itemize}
This generalizes \S $\ref{ex:chebyshev}$,  up to the transformation $x\mapsto x/(x-1)$
and a trigonometric substitution.

The $\ZZ/2\ZZ$ monodromy is the special case $m=1$. 
Inspection of Riemann's $P$-symbols for the second order
equations for $\Theta_1(x)$, $\Theta_2(x)$ at the end of $\cite[\S 5.3]{DihedralTR}$
shows that the roots of $\Psi(x)$ are not singularities then.
The equations are then hypergeometric, and we can identify % (up to constant multiples):
\begin{align} \label{eq:dihedm1}
\Theta_1(x)&\cong \hpg21{-\frac{d}2,\ell+\frac{1-d}2}{\frac12-k}{x},\\
\label{eq:dihedm1a}
\Theta_2(x)&\cong \hpg21{k+\ell+1-\frac{d}2,k+\frac{1-d}2}{\frac32+k}{x},
\end{align}
with the degree $d=k+\ell+n$. 
This is consistent with the expressions of $\Theta_1(x),\Theta_2(x)$
as Jacobi polynomials here.

Jacobi polynomials can be identified in other dihedral case $k=1$, $\ell=0$ (and any $m$).
Then hypergeometric expressions in \cite[\S 5.2]{DihedralTR} give
\begin{align*}
\Theta_1(x)\cong P_{\lceil m/2 \rceil}^{(-3/2,-m)}\!\left(1-\frac{2x}{m^2}\right),\quad
\Theta_2(x)\cong %\frac{m^2-1}{3m^2}\,
P_{\lfloor m/2 \rfloor-1}^{(3/2,-m)}\!\left(1-\frac{2x}{m^2}\right).
\end{align*}
The dessins d'enfant are depicted in Figure \ref{fg:jacobi}~\refpart{d}--\refpart{e}.
\end{remark}

Jacobi polynomials appear in the same way in pull-back transformations of hypergeometric
equations with other finite cyclic monodromies. Those pull-back transformations are implied
by Klein's theorem for second order Fuchsian equations with finite monodromy \cite{kleinvhw}.
As a special case, the hypergeometric equations with trivial monodromy
are $E(\ell,n,m)$ with odd $k+\ell+m$ and the triangle equalities $n<k+\ell$, $k<\ell+n$,
$\ell<k+n$ satisfied \cite[\S 8]{degeneratehpg}. 
The Kleinian Belyi covering for $E(\ell,n,m)$ is defined by the identity
\begin{align*}
P_{M-1}^{(-k,-\ell)}(2x-1)=
(-1)^{k+1}\,x^\ell P_K^{(-k,\ell)}(2x-1) + (1-x)^k P_L^{(k,-\ell)}(2x-1).
\end{align*}
Here $K=(k-\ell+n-1)/2$, $L=(\ell-k+n-1)/2$, $M=(k+\ell-n+1)/2$.
The dessin d'enfant is depicted in Figure \ref{fg:jacobi}~\refpart{f}.

Kleinian transformations for hypergeometric equations with the $\ZZ/m\ZZ$ monodromy 
are obtained from the solution basis in \cite[(42)]{degeneratehpg}.
The Belyi covering is the quotient of the following two functions raised to the $m$th power:
\begin{equation}
x^{p/m} P_K^{(p/m,-n)}(1-2x), \qquad
P_L^{(-p/m,-n)}(1-2x).
\end{equation}
Here $n=K+L+1$, and $p\in\ZZ\setminus m\ZZ$.  
The local exponent differences are $n,p/m,K-L+p/m$.
If we express the hypergeometric equation as $E(p/m,q/m,n)$ 
with positive integers $p,q,m$ such that $p,q\not\in m\ZZ$,
then either $p-q$ or $p+q$ must be divisible by $m$. 
Let $r$ be the integer in $\{(p-q)/m,(p+q)/m\}$;
then $n+r$ must be odd because $K=(n-r-1)/2$, $L=(n+r-1)/2$. 
% \begin{equation}
% (-1)^{n+m}\,x^{p/q} P_n^{(-n-m-1,p/q)}(2x-1), \qquad
% P_m^{(-n-m-1,-p/q)}(2x-1).
% \end{equation}

The dessins d'enfant for $m=3$ are depicted in Figure \ref{fg:jacobi}~\refpart{g}--\refpart{j}.
There $M''=(p+q+m(1-n))/2$, $k=\lfloor p/m \rfloor$, $\ell=\lfloor q/m \rfloor$, 
and $N=(n-k-\ell)/2$,  $N'=(n-k-\ell-1)/2$ as above. 
To describe the dessins for larger $m$, let us call an {\em $s$-bridge} 
a sequence of exactly $s$ edges connecting a pair of black and white vertices. For example,
the dessin \refpart{g} can be described as having some number of $3$-bridges at both ends and
an $M''$-bridge in between. The dessins \refpart{h}--\refpart{j} have an alternative sequence 
of 1- and 2-bridges between the ends. The dessins for larger $m$ have either the shape like 
dessin \refpart{g} with $m$-bridges (rather than 3-bridges) at both ends, 
or look like dessins \refpart{h}--\refpart{j} 
with end blossoms of $m$-bridges connected by an alternative sequence of $s$- and $t$-bridges
with $s+t=m$. The dessins with $k=\ell=0$ are just sequences of alternating $s$- and $t$-bridges,
like in Figure \ref{fg:jacobi}~\refpart{k}.
Then either $p+q=m$ or $p=q<m$. The pull-back Belyi function can then be expressed in terms of
properly orthogonal (i.e., $\alpha,\beta>-1$) Jacobi polynomials. To get the Belyi function,
we take the $m$-th power of
\begin{equation}
\frac{x^{p/m}\,P_{n/2-1}^{(p/m,\,q/m)}(\frac{1+x}{1-x})}
{(1-x)\,P_{n/2}^{(-p/m,-q/m)}(\frac{1+x}{1-x})} \quad\mbox{or}\quad 
\frac{x^{p/m}\,P_{(n-1)/2}^{(p/m,-p/m)}(\frac{1+x}{1-x})}
{P_{(n-1)/2}^{(-p/m,\,p/m)}(\frac{1+x}{1-x})}
\end{equation}
for $p+q=m$ or $p=q<m$, respectively.

\small

\bibliographystyle{plain}
\bibliography{AGSF}

\end{document}